%% file: attributed-c-sets.tex
\documentclass[ a4paper, onecolumn, superscriptaddress, 10pt, accepted=2022-02-27, issue=5, volume=4, shorttitle=papers/compositionality-4-5 ]{compositionalityarticle}
\pdfoutput=1

\usepackage[utf8]{inputenc}
\usepackage[T1]{fontenc}

\usepackage{macros}

\usepackage[frozencache]{minted}
\setminted{xleftmargin=1em, escapeinside=||}

\title{Categorical Data Structures for Technical Computing}
\date{}
\author{Evan Patterson}
\email{evan@epatters.org}
\homepage{https://www.epatters.org}
\affiliation{Topos Institute, California, USA}
\author{Owen Lynch}
\email{o.c.lynch@students.uu.nl}
\homepage{https://owenlynch.org}
\affiliation{Universiteit Utrecht, Mathematics Department, Utrecht, The Netherlands}
\author{James Fairbanks}
\email{fairbanksj@ufl.edu}
\homepage{https://jpfairbanks.com}
\affiliation{University of Florida, Computer \& Information Science \& Engineering, Florida, USA}

\newcommand*{\C}{\mathsf{C}}

\newcommand*{\D}{\cat{D}}

\newcommand*{\jul}[1]{\mintinline{julia}{#1}}
\newcommand*{\pres}[1]{\langle #1 \rangle}

\DeclareMathOperator{\suc}{succ}

\newcommand*{\sto}{\shortrightarrow}
\newcommand*{\acset}{acset}
\newcommand*{\Acset}{Acset}

\DeclareUnicodeCharacter{22A3}{\ensuremath{\dashv}}

\begin{document}
\maketitle

\begin{abstract}
  Many mathematical objects can be represented as functors from finitely-presented categories $\C$ to $\setC$. For instance, graphs are functors to $\setC$ from the category with two parallel arrows. Such functors are known informally as $\C$-sets. In this paper, we describe and implement an extension of $\C$-sets having data attributes with fixed types, such as graphs with labeled vertices or real-valued edge weights. We call such structures \emph{acsets}, short for \emph{attributed $\C$-sets}. Derived from previous work on algebraic databases, acsets are a joint generalization of graphs and data frames. They also encompass more elaborate graph-like objects such as wiring diagrams and Petri nets with rate constants. We develop the mathematical theory of {\acset}s and then describe a generic implementation in the Julia programming language, which uses advanced language features to achieve performance comparable with specialized data structures.
\end{abstract}

\tableofcontents

\section{Introduction}

Practicing data scientists commonly say that they spend at least half their time cleaning data and managing data pipelines, rather than fitting models. The inherent difficulties of data preparation are exacerbated by the multitude of ways in which data can be stored. Data can live in SQL databases, data frames, graphs, or any of the more specialized data structures scattered throughout the data science ecosystem. While each data structure may have its own advantages, it also comes with a new programming interface to learn, and this proliferation fragments the ecosystem and creates challenges for interoperability.

One solution to this problem has been to centralize around a single data structure: the data frame. A data frame is a column-oriented data table, which, unlike a matrix, can have columns with different types. The Python ecosystem has the popular package \texttt{pandas} \cite{mckinney_data_2010}; R has the built-in data type \texttt{data.frame} \cite{r_core_team_r_2020} and its cousin \texttt{tibble} \cite{muller_tibble_2018}. In Julia, many different packages implement the common interface of \texttt{Tables.jl}, which helps somewhat with interoperability.

Underlying all these packages, however, is the same abstract data model, with all its limitations. Importantly, the data frame model fails to capture \emph{relations} between entities, a concept realized in SQL as the \texttt{FOREIGN KEY}. Of course, it is possible to maintain a collection of data frames that refer to each other according to an ad hoc convention, but such relationships are not formalized and thus cannot be manipulated conveniently or robustly through high-level abstractions. As a prototypical example, data frames do not properly capture the structure of a graph, which requires two interlinked data tables, one for vertices and one for edges.

A well-known abstraction that encompasses both data frames and graphs is the relational database. Relational databases have schemas that describe a collection of relations and the foreign keys that link them together. However, relational databases tend to be monolithic systems that are difficult to integrate with general-purpose programming languages. Relational query languages like SQL and logic programming languages like Prolog and Datalog suffer from the Achilles' heel of all standalone domain-specific languages (DSLs): one cannot step outside the DSL without abandoning the system entirely. Yet stepping outside the DSL is often necessary, since even the most expressive query languages have their limitations. Therefore, it becomes a perennial question how much logic to put into the query and how much to put into postprocessing outside the database: a ``two-language problem'' for data processing. Furthermore, databases are typically designed to be treated as global mutable state, making it unnatural to use them as disposable objects in a localized context. For example, graphs can be modeled as databases, but while a program could easily have thousands or millions of graphs extant at a given time, existing in different scopes, one would not wish to achieve this using a million SQL database instances.

Yet there is no essential reason that the idea of ``a collection of tables linked together by a schema, along with indices'' should require monolithic systems, persistence, or global scope. The earliest papers on relational databases already possessed a mathematical model, based on relational algebra and first-order logic, that was independent of any conventions about implementation \cite{codd_relational_1970}. The functional data model, whose basic constructs are entities and functions, was introduced not long afterwards \cite{shipman1981}. In the functional model, functions are primary and relations are replaced by spans of functions, called tabulating spans. Johnson et al formulated a version of this data model category-theoretically, using sketches \cite{johnson2002}. More recently, Spivak recognized that the functional data model can be elegantly reconstructed as a functor from a finitely-presented category to the category of sets and functions \cite{spivak_functorial_2012}. Extensions of Spivak's functorial model to accommodate data attributes were developed later \cite{spivak_wisnesky_2015,schultz_algebraic_2016}.

In this paper, we present an efficient in-memory implementation of categorical databases. Depending on the schema, the resulting data structure can act as a data frame, a graph, or any of a multitude of other structures, previously regarded as too niche to have dedicated implementations. Using advanced features of the Julia programming language, our implementation achieves performance competitive with state-of-the-art graph libraries, despite the fact that our graph library is a thin wrapper around a much more general system.

We call our data structures ``{\acset}s,'' an abbreviation of ``attributed $\C$-sets.'' We define {\acset}s in a direct, practical manner but show that they can be reformulated as well-understood objects from categorical algebra. From this mathematical picture, we derive high-level operations to combine multiple {\acset}s together, query {\acset}s, and translate between {\acset}s based on different schemas. For instance, we can generically compute finite limits and colimits of acsets on a fixed schema. We emphasize, however, that for most purposes, knowledge of category theory is not required to use acsets.

Our impetus for developing {\acset}s originated with the needs of Catlab and other packages in the AlgebraicJulia ecosystem \cite{evan_patterson_2021_5771194}. While implementing various pieces of applied category theory, we realized that many of the data structures we needed were captured, at least partially, by $\C$-sets (copresheaves), which we could implement generically. But this abstraction was not completely satisfactory, because it did not account for attributes: data with a fixed, external meaning such as real numbers or strings of text. This eventually led us to the formalism of {\acset}s and to a more systematic software implementation.

\paragraph{Contributions} Our main contribution is an efficient, flexible implementation of categorical databases as in-memory data structures in a general-purpose programming language, supporting key constructions of applied category theory, including decorated or structured cospans. Our implementation takes advantage of metaprogramming features of the Julia language to attain performance comparable with specialized, state-of-the-art graph libraries, while being far more general. This combination of performance and generality in categorical data structures appears to be novel. In a more theoretical vein, we further explore the design space for categorical databases, introducing a variant of Spivak et al's functorial data model, intermediate in complexity between those in \cite{spivak_wisnesky_2015} and \cite{schultz_algebraic_2016}, and deriving its basic mathematical properties.

\paragraph{Outline} The paper begins, in \cref{section:practice}, with an informal overview of {\acset}s that should be accessible to a general audience of computer scientists and software engineers. We then review the mathematical theory of $\C$-sets, of which {\acset}s are an elaboration, in \cref{section:csets}. Readers familiar with the relevant mathematics may omit this section. In \cref{section:theory}, we develop the theory of attributed $\C$-sets, showing that {\acset}s are slice objects in the category of $\C$-sets and deriving consequences of this fact. Finally, in \cref{section:implementation}, we discuss the implementation of {\acset}s and benchmark it against \texttt{LightGraphs.jl}, a state-of-the-art graph library written in Julia \cite{bromberger_lightgraphsjl_2017}, lending empirical support to our claim that {\acset}s can simultaneously achieve generality and performance.

\section{Using Attributed \texorpdfstring{$\C$}{C}-sets} \label{section:practice}

In this section, we aim to convey an intuitive understanding of {\acset}s. We first explain how two common data structures, data frames and graphs, are special cases of {\acset}s. To illustrate the breadth of the formalism, we also give a short tour of more exotic, yet useful, {\acset}s.

\subsection{Data Frames and Graphs}

As discussed in the Introduction, data frames are a popular answer to the question of how we should store our data. \cref{tab:ex_data frame} shows a tiny data frame with two columns, \texttt{a} and \texttt{b}.

\begin{table}[ht!]
  \centering
  \begin{tabular}{c|c}
    \texttt{a} & \texttt{b} \\\hline
    3  & 0.2 \\
    2  & 0.0 \\
    2  & 0.0 \\
    1  & 0.9
  \end{tabular}
  \caption{Example of a data frame}
  \label{tab:ex_data frame}
\end{table}

For the purposes of this paper, we regard a \emph{data frame} to be a collection of 1-dimensional arrays, called \emph{columns}, all of the same length. A \emph{row} in a data frame consists of the values of all columns at a given integer index. Individual columns can be retrieved by name, and individual rows can be retrieved by index. Data frames are typically stored \emph{column-wise} (as a list of columns) rather than \emph{row-wise} (as a list of rows) to permit efficient data access and iteration.

We will build {\acset}s as an extension of data frames, so we take for granted an implementation of data frames in Julia. In this notional implementation, \jul{DataFrame{(:a,:b),Tuple{Int,Float64}}} is the type of a data frame with two columns named \jul{:a} and \jul{:b} and having the types \jul{Int} and \jul{Float64}. We can construct and inspect the above data frame as follows, accessing the data both row-wise and column-wise.

\begin{minted}{julia}
  > df = DataFrame(a=[1,2,2,3], b=[0.2,0.,0.,0.9]);
  > df.a # first column, called `a`
  [1,2,2,3]
  > df[2,:] # second row
  (a=2,b=0.)
\end{minted}

As motivated in the Introduction, we will eventually wish to regard a graph as two interlinked data frames. However, before turning to that strategy, we consider how graphs are typically implemented. Perhaps the simplest way to implement a graph in Julia would be to use the \emph{edge list} data structure:
\begin{minted}{julia}
struct EdgeList
  vertices::Int
  edges::Int
  src::Vector{Int}
  tgt::Vector{Int}
end
\end{minted}
An object \jul{g} of type \jul{EdgeList} defines a graph, where
\begin{enumerate}
  \item the vertices in the graph are the consecutive numbers \jul{1:g.vertices},
  \item the edges in the graph are the consecutive numbers \jul{1:g.edges}, and
  \item the source of edge \jul{j} is \jul{g.src[j]} and the target of edge \jul{j} is \jul{g.tgt[j]}.
\end{enumerate}
In particular, the invariant \jul{length(g.src) == length(g.tgt) == g.edges} should be maintained. Just as for data frames, we store \jul{src} and \jul{tgt} as two separate vectors, rather than as a single vector of pairs.

A common variant of the edge list is the \emph{adjacency list}. It stores the inverse images of the source and target maps, which, to use the jargon of databases, index the edges incoming to or outgoing from each vertex.
\begin{minted}{julia}
struct AdjacencyList
  vertices::Int
  edges::Int
  src_index::Vector{Vector{Int}}
  tgt_index::Vector{Vector{Int}}
end
\end{minted}
The interpretation of an object \jul{g} of type \jul{AdjacencyList} is that for each vertex \jul{i}, the list of edges with source \jul{i} is \jul{g.src_index[i]} and the list of edges with target \jul{i} is \jul{g.tgt_index[i]}. These indices are useful for traversing the graph, as they enable rapid iteration through the neighbors of any given vertex.

However, in the adjacency list format, it is less convenient to iterate through the edges of the graph and retrieve their sources and targets. Therefore, it can be useful to have both the \jul{src} and \jul{src_index} fields. The trouble now becomes that when modifying the graph, careful bookkeeping is required to ensure that all the fields remain consistent with each other. A useful feature of our implementation of {\acset}s is automatically generating code to deal with such indices. Correct and efficient implementation of this bookkeeping is a significant task when implementing new mutable data structures.

\subsection{Towards Attributed \texorpdfstring{$\C$}{C}-sets} \label{subsection:toward_acsets}

There is an essential difference between graphs and data frames. One can permute the vertices of a graph, and as long as the source and target maps are updated accordingly, the meaning of the graph does not change. However, in a data frame, if one were to exchange every occurrence of the number 6.0 with the number 2.0, then the meaning of the data has changed drastically.

This distinction is crucial to understanding {\acset}s, so it is worth introducing some informal terminology. We call the connectivity data stored in a graph ``combinatorial data'' and the kind of data stored in a data frame ``attribute data.'' Distinctions of this nature are standard in database theory. For example, Johnson et al's EA sketches distinguish between entities and attributes \cite{johnson2002}, with entities being what we call combinatorial data.

In data science and scientific computing, we routinely encounter datasets involving both combinatorial and attribute data. For instance, suppose that we want to represent a network of roads. We could store a graph where each vertex (road junction) has associated $x$ and $y$ coordinates, and each edge (road) has a length, which could be different from the Euclidean distance between its endpoints due to bends and hills. A straightforward implementation of the corresponding data structure might be:
\begin{minted}{julia}
struct MessyRoadMap
  vertices::Int
  edges::Int
  coords::DataFrame{(:x,:y), Tuple{Float64,Float64}}
  lengths::DataFrame{(:length,), Tuple{Float64}}
  src::Vector{Int}
  tgt::Vector{Int}
  src_index::Vector{Vector{Int}}
  tgt_index::Vector{Vector{Int}}
end
\end{minted}
Organizing the fields more systematically, in our implementation of {\acset}s this data structure appears as:
\begin{minted}{julia}
struct OrganizedRoadMap{T}
  tables::NamedTuple{(:vertex,:edge),
    Tuple{
      DataFrame{(:x,:y), Tuple{T,T}},
      DataFrame{(:src,:tgt,:length), Tuple{Int,Int,T}}}}
  indices::NamedTuple{(:vertex,:edge),
    Tuple{
      DataFrame{(:src,:tgt), Tuple{Vector{Int},Vector{Int}}},
      DataFrame{(),Tuple{}}}}
end
\end{minted}
If \jul{i} is the index of a vertex in an object \jul{g} of this reorganized type, then to access \jul{i}'s $x$-coordinate, we write \jul{g.tables.vertex.x[i]}. Similarly, the source of an edge \jul{j} is \jul{g.tables.edge.src[j]}. The edges incoming to a vertex \jul{i} are given by \jul{g.indices.vertex.tgt[i]}.

This way of organizing the data is evidently not at all specific to road maps. The relevant features of the road map are:
\begin{enumerate}
  \item There are two ``combinatorial types'': vertices and edges. We write these in mathematical notation as $V$ and $E$. In an instantiation of a graph, we will assign finite sets to $V$ and $E$.
  \item There are two maps from edges to vertices, which we write as $\src \maps E \to V$ and $\tgt \maps E \to V$.
  \item There is a single ``attribute type'' $T$. In an instance of a road map, we will assign a Julia type to $T$, typically some numeric type.
  \item There are ``data attribute'' maps $x \maps V \to T$, $y \maps V \to T$, and $\mathrm{length} \maps E \to T$.
  \item The maps $\src$ and $\tgt$ are indexed.
\end{enumerate}
Items 1 and 2 comprise the combinatorial data of the road map and items 3 and 4 the attribute data. In our implementation, combinatorial data is always represented by integers, whereas attribute data is represented by type parameters and may be filled by any Julia type. Note that the final item 5 is not like the others; indices are required for efficiency, but if we had omitted them, the data structure would still contain the same logical information. For this reason, when we describe signatures formally, we will omit indices. Instead, the indices can be chosen at compile time based on the anticipated workload. Only keys that are frequently queried should be indexed.

In our implementation in Catlab, we generate a data type like \jul{OrganizedRoadMap} by writing down a formal specification of the schema and then passing this specification to a function that programatically generates the data structure for the corresponding {\acset}:

\noindent\begin{minipage}[t]{\textwidth}
  \vspace{0.5em}
  \begin{minipage}[t]{0.5\textwidth}
\begin{minted}{julia}
@present TheoryRoadMap(FreeSchema) begin
  (V,E)::Ob
  (src,tgt)::Hom(E,V)

  T::AttrType
  (x,y)::Attr(V,T)
  length::Attr(E,T)
end

@acset_type RoadMap(TheoryRoadMap, index=[:src,:tgt]) # creates RoadMap type
\end{minted}
  \end{minipage}%
  \begin{minipage}[t]{0.5\textwidth}
\[\begin{tikzcd}
	E && V \\
	\\
	& {\mathtt{T}}
	\arrow["{\mathrm{tgt}}"', curve={height=6pt}, from=1-1, to=1-3]
	\arrow["{\mathrm{src}}", curve={height=-6pt}, from=1-1, to=1-3]
	\arrow["{\mathrm{y}}", curve={height=-6pt}, from=1-3, to=3-2]
	\arrow["{\mathrm{x}}"', curve={height=6pt}, from=1-3, to=3-2]
	\arrow["{\mathrm{length}}"{description}, from=1-1, to=3-2]
\end{tikzcd}\]
  \end{minipage}%
  \vspace{1em}
\end{minipage}
The elements in the schema typed by \jul{Ob} and \jul{Hom} describe the combinatorial data and the elements typed by \jul{AttrType} and \jul{Attr} describe the attribute data. The schema is also displayed as a diagram to the right of the code.

Graphs and data frames can now be treated as special cases of this machinery, one involving only combinatorial data and the other involving only attribute data. The schema for graphs is:

\noindent \begin{minipage}[t]{\textwidth}
  \vspace{0.5em}
  \begin{minipage}[t]{0.5\textwidth}
\begin{minted}{julia}
@present TheoryGraph(FreeSchema) begin
  (V,E)::Ob
  (src,tgt)::Hom(E,V)
end
\end{minted}
  \end{minipage}%
  \begin{minipage}[t]{0.5\textwidth}
    \[\begin{tikzcd}
        E \ar[rr, "\mathrm{tgt}"', curve={height=6pt}] \ar[rr, "\mathrm{src}", curve={height=-6pt}] & & V
      \end{tikzcd} \]
  \end{minipage}%
  \vspace{1em}
\end{minipage}
The schema for \jul{DataFrame{(:a,:b),Tuple{T,T}}}, a data frame with two columns of the same type, is:

\noindent\begin{minipage}[t]{\textwidth}
  \vspace{0.5em}
  \begin{minipage}[t]{0.5\textwidth}
\begin{minted}{julia}
@present TheoryAB(FreeSchema) begin
  Sample::Ob
  T::AttrType
  (a,b)::Attr(Sample,T)
end
\end{minted}
  \end{minipage}%
  \begin{minipage}[t]{0.5\textwidth}
    \vspace{-1.5em}
    \[\begin{tikzcd}
        \mathbf{Sample} \ar[dd, "\mathrm{a}"', curve={height=6pt}] \ar[dd, "\mathrm{b}", curve={height=-6pt}]\\
        \\
        \mathtt{T}
      \end{tikzcd} \]
  \end{minipage}%
  \vspace{1em}
\end{minipage}

Defining data structures is all well and good, but how does one use them? To illustrate, we show a simple function that generates a road map whose underlying graph is a path graph.

\begin{minted}{julia}
function make_path(coords::Vector{Tuple{Float64, Float64}})
  # Create an empty roadmap
  path = RoadMap{Float64}()

  # This is a convenient function that calculates the Euclidean distance between two
  # vertices in the road map. Notice that we can reference attributes using indexing
  # and that the system knows that these attributes belong to vertices, not edges.
  dist(i,j) = sqrt((path[i,:x] - path[j,:x])^2 + (path[i,:y] - path[j,:y])^2)

  x, y = coords[1]
  # add_part! mutates path to add a part, returning the index of the added part.
  # The named arguments to this function assign the attributes of that part.
  src = add_part!(path, :V, x=x, y=y)

  for i in 2:length(coords)
    x, y = coords[i]
    tgt = add_part!(path, :V, x=x, y=y)
    add_part!(path, :E, src=src, tgt=tgt, length=dist(i,j))
    src = tgt
  end
  path
end
\end{minted}

While there are also higher-level functions on {\acset}s, the low-level accessors and mutators are always available. Moreover, they are \emph{fast}, as demonstrated by the benchmarks in \cref{section:implementation}, so the user is not constrained by what the high-level interface exposes. Coupled with the ability of Julia to make hand-written loops as efficient as ``vectorized'' code, users can easily write high-performance algorithms that are unanticipated by the core library.

\subsection{Beyond Graphs: Wiring Diagrams and Other Graph-like Structures}

Objects similar to graphs but possessing extra or different structure occur frequently in computer science. To create custom data structures for each graph variant by hand would cause an explosion of software complexity, yet without custom data structures, it is not possible to efficiently utilize the extra mathematical structure. In this section, we show that many graph-like structures are unified by the concept of an {\acset} and thus can be manipulated through a uniform, general software interface.

Each of the following three examples are accompanied by figures, which are too large to be displayed inline. The reader is encouraged to contemplate the figures spread over the next several pages before returning to the main text for an explanation.

\begin{figure}
  \centering
  \input{figures/port-graphs.tex}
  \caption{Port graphs}
  \label{fig:port_graphs}
\end{figure}
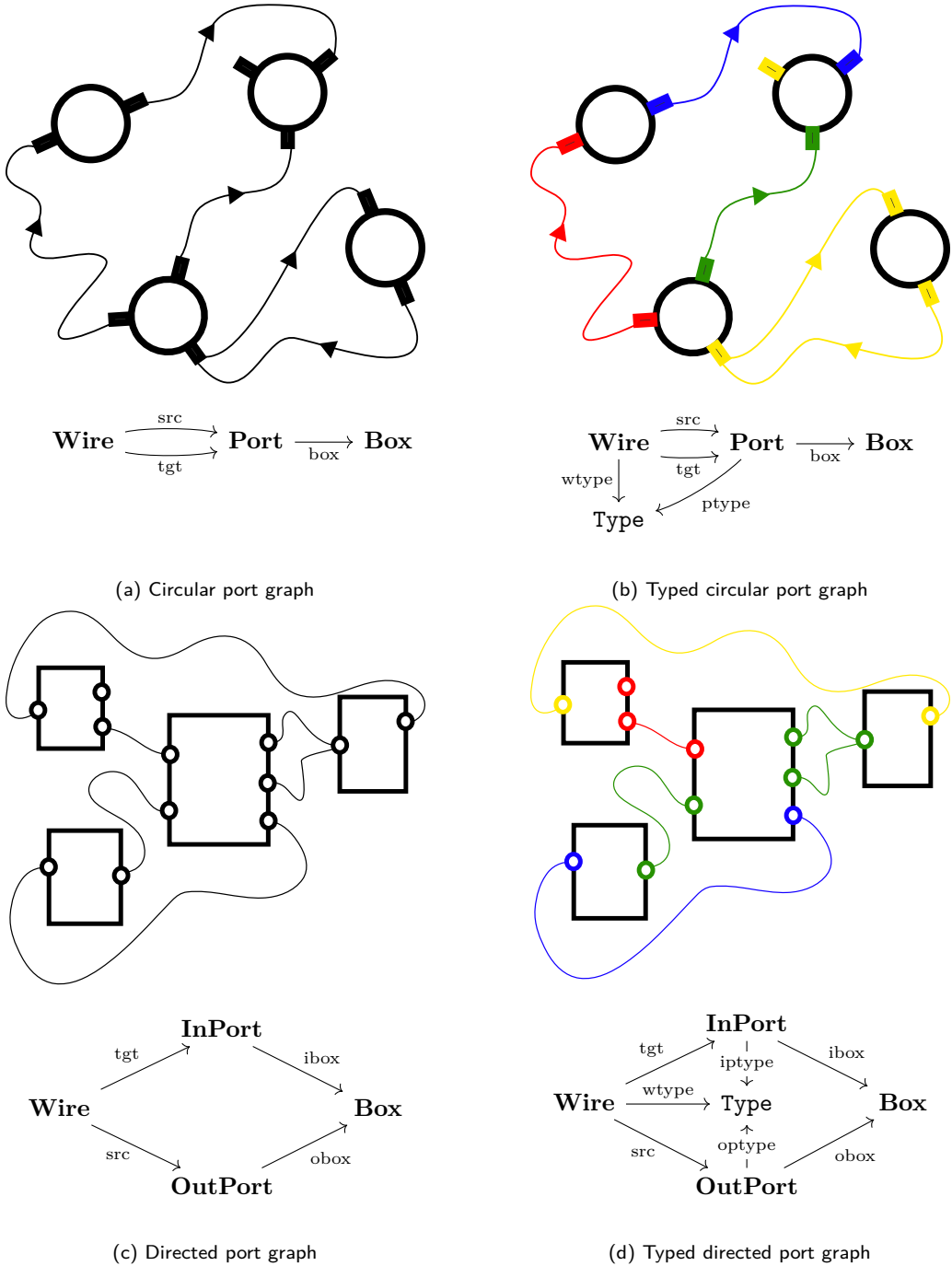

\begin{figure}
  \centering
  \input{figures/petri-nets.tex}
  \caption{Whole-grained Petri nets}
  \label{fig:petri_nets}
\end{figure}
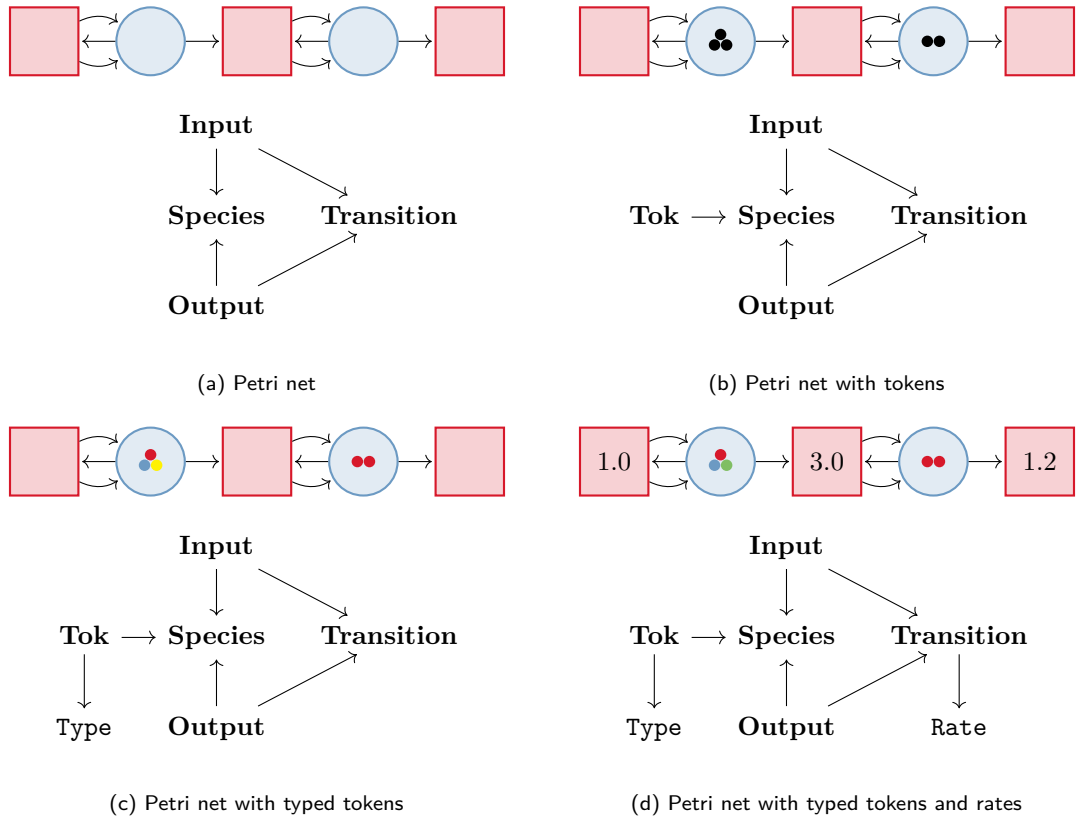

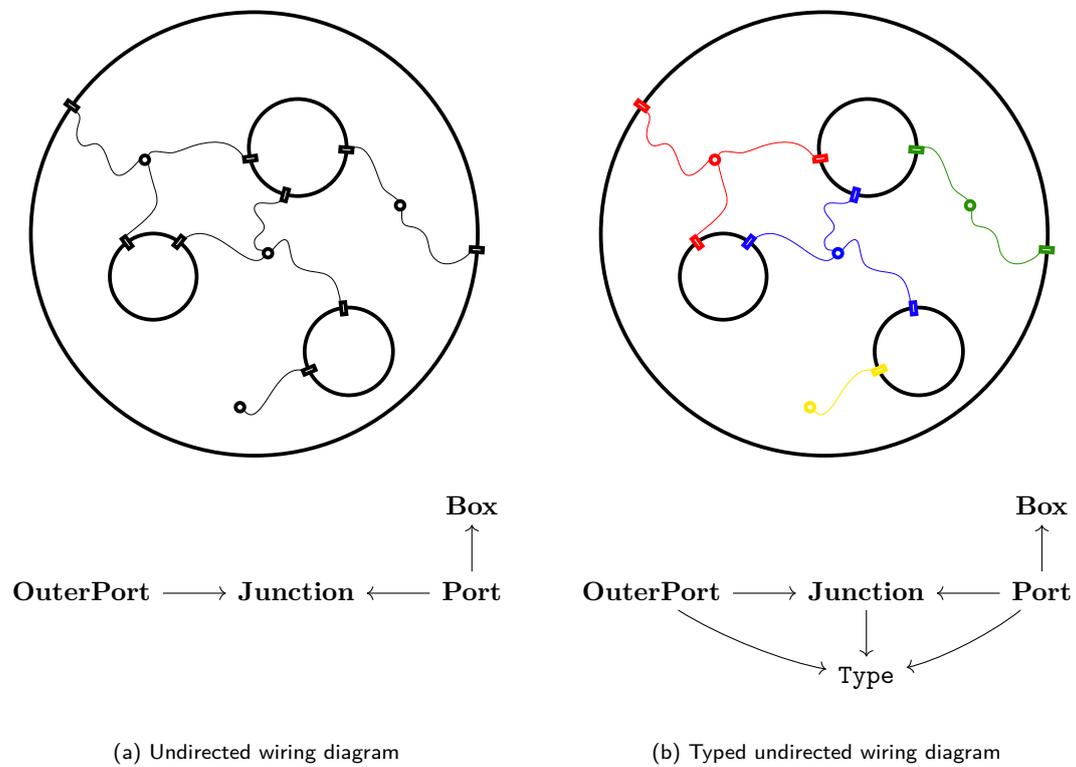
\begin{figure}
  \centering
  \input{figures/undirected-wiring-diagrams.tex}
  \caption{Undirected wiring diagrams}
  \label{fig:undirected_wiring_diagrams}
\end{figure}

\begin{example}
  \cref{fig:port_graphs} shows four different kinds of graphs with ports, or ``port graphs,'' and the generators of their schemas. The port graphs differ along two axes: untyped versus typed and circular versus directed. In a typed port graph, the types of incident ports and wires must agree. This requirement is expressed by the equation
  \[ \mathrm{src} \cmp \mathrm{ptype} = \mathrm{wtype} = \mathrm{tgt} \cmp \mathrm{ptype} \]
  for circular port graphs, and
  \[ \mathrm{src} \cmp \mathrm{optype} = \mathrm{wtype} = \mathrm{tgt} \cmp \mathrm{iptype} \]
  for directed port graphs, which assert that certain triangles in \cref{subfig:typed_port_graph,subfig:typed_directed_port_graph} commute. Untyped port graphs have no such requirement. In a circular port graph, the boxes have only kind of port, $\mathbf{Port}$, although the wires are directed. In a directed port graph, the ports are split into input ports ($\mathbf{InPort}$) and output ports ($\mathbf{OutPort}$), and the wires must go from input ports to output ports. By convention, input ports are drawn on the left and output ports on the right, so that the directions of the wires can be inferred from the incident ports.
\end{example}

\begin{example} \label{ex:petri_nets}
  A \emph{whole-grained Petri net} consists of species, transitions, inputs to transitions, and outputs from transitions \cite{kock_elements_2020}. We visualize Petri nets in \cref{fig:petri_nets}, where, by tradition, the species are drawn as circles and the transitions as squares. Petri nets are often augmented with a set of tokens for each species. This is accomplished by adding a new object $\mathbf{Tok}$ to the schema, along with a map from tokens to species. This trick of representing a many-to-one relationship via a map is a common one when modeling data with relational algebra. Note that the schema for a whole-grained Petri net is isomorphic to the schema for a directed bipartite graph, to which we return in \cref{example:bipartite_graph}.
\end{example}

\begin{example}
  Schemas for typed and untyped \emph{undirected wiring diagrams} are pictured in \cref{fig:undirected_wiring_diagrams}. Loosely speaking, undirected wiring diagrams represent patterns of composition for systems with an outer boundary. Imagine placing an entire wiring diagram inside one of the inner circles, and then erasing the inner circle to get a new wiring diagram. This operation makes undirected wiring diagrams into an operad, a construction made precise in \cite{rupel_operad_2013}.
\end{example}

We hope that these examples have convinced the reader that expressing graph-like data structures as acsets is both natural and general. The remainder of this paper will be mainly concerned with the mathematics and implementation of {\acset}s, so the reader who is primarily interested in the practical use of this technology may proceed directly to the software in Catlab.

\section{Review of \texorpdfstring{$\C$}{C}-sets} \label{section:csets}

Before turning to the theory of attributed $\C$-sets, we review the concept of a $\C$-set, also known as a copresheaf on a category $\C$. In this section and the next, we assume acquaintance with the basic concepts of category theory, namely categories, functors, and natural transformations.

\subsection{Definition and Examples}

\begin{definition}
  Given a small category $\C$, a \emph{$\C$-set} is a functor $F \maps \C \to \setC$. For the purposes of computing, we sometimes restrict to \emph{finite} $\C$-set, which are functors $F \maps \C \to \finSetC$.
\end{definition}

\begin{example}
  Let $\cat{Gr} = \{E \rightrightarrows V\}$ be the category with two objects $E$ and $V$ and two parallel morphisms, $\src,\tgt \maps E \to V$. Then a $\cat{Gr}$-set is a graph.
\end{example}

When implementing mathematics on a computer, everything must represented by a finite amount of data. In the case of $\C$-sets, this means that we work only with \emph{finitely presented} categories $\C$. The next several definitions explain this idea, starting with the simplest way to present a category, as a free category.

\begin{definition}
  There is a forgetful functor $\abs{-} \maps \catC \to \setC^{\cat{Gr}}$ that takes a small category to its underlying graph. We define $\pres{-} \maps \setC^{\cat{Gr}} \to \catC$ to be the left adjoint to this functor, which takes a graph $G$ to the \emph{free category} $\pres{G}$ on that graph. Concretely, an object of $\pres{G}$ is a vertex of $G$ and a morphism $u \to v$ of $\pres{G}$ is a (possibly empty) path in $G$ from vertex $u$ to vertex $v$.
\end{definition}

\begin{definition}
  A \emph{parallel pair} in a category $\C$ is a diagram in $\C$ of shape $\{0 \rightrightarrows 1\}$, or equivalently a pair of morphisms $(f,g)$ in $\C$ such that $\dom f = \dom g$ and $\codom f = \codom g$.
\end{definition}

In the following definition, parallel pairs will represent the equations between morphisms that are imposed in a finitely presented category.

\begin{definition}
  If $G$ is a graph and $R$ is a set of parallel pairs in $\pres{G}$, let $\pres{G|R}$ be the initial category equipped with a map $q_{R} \maps \pres{G} \to \pres{G|R}$, called the \emph{quotient map}, such that $q_{R}(f) = q_{R}(g)$ for all $(f,g) \in R$. We say that the category $\pres{G|R}$ is \emph{presented} by the \emph{generators} in $G$ and \emph{relations} in $R$. In particular, a category $\C$ is \emph{finitely presented} by a graph $G$ and relations $R$ if $G$ and $R$ are both finite and $\C \cong \pres{G|R}$.
\end{definition}

To say that $\pres{G|R}$ is the \emph{initial category} with a quotient map $q_{R} \maps \pres{G} \to \pres{G|R}$ is to say that for any other category $\C$ with a map $q'_{R} \maps \pres{G} \to \C$ such that $q_{R}'(f) = q_{R}'(g)$ for all $(f,g) \in R$, there exists a unique functor $F \maps \pres{G|R} \to \C$ such that $q_{R}' = q_{R} \cmp F$. An explicit construction of the category $\pres{G|R}$ is given by Borceux \cite[Proposition 5.1.6]{borceux1994a}, among other sources.

\begin{example} \label{example:symmetric_graph}
  Consider the graph $G$ given by
  \[
    \begin{tikzcd}
      E \ar[r,shift left,"\src"] \ar[r,shift right,swap,"\tgt"] \ar[loop left,"\inv"] & V
    \end{tikzcd}
  \]
  and the set of parallel pairs
  \[ R = \set{(\inv^{2},1_{E}), (\inv \cmp \src,\tgt), (\inv \cmp \tgt, \src)}. \]
  Then $\cat{SymGr} = \pres{G|R}$ is a finitely presented category such that $\cat{SymGr}$-sets are \emph{symmetric graphs}. In a symmetric graph $F$, every edge $e \in F(E)$ is paired with an edge $F(\inv)(e)$ going in the opposite direction, so that the set $\{e, F(\inv)(e)\}$ can be interpreted as an  ``undirected edge'' between $F(\src)(i)$ and $F(\tgt)(i)$.
\end{example}

\begin{example} \label{example:bipartite_graph}
  Let $\cat{BipartiteGr}$ be the category freely generated by the graph
  \[
    \begin{tikzcd}
      & V_{a} & \\
      E_{ab} \ar[ur,"\src_{a}"] \ar[dr,"\tgt_{b}",swap] & & E_{ba} \ar[ul,swap,"\tgt_{a}"] \ar[dl,"\src_{b}"] \\
      & V_{b}
    \end{tikzcd}.
  \]
  Then a $\cat{BipartiteGr}$-set $F$ is a \emph{bipartite graph} where $F(V_{a})$ and $F(V_{b})$ are the vertices in each partition and $E_{ab}$ and $E_{ba}$ are the edges going from $V_{a}$ to $V_{b}$ and from $V_{b}$ to $V_{a}$, respectively.
\end{example}

\begin{example}
  Let $\cat{Dyn}$ be the category freely generated by
  \[
    \begin{tikzcd}
      X \ar[loop right,"\suc"]
    \end{tikzcd}.
  \]
  Then a $\cat{Dyn}$-set is a discrete dynamical system. When viewed as a monoid, $\cat{Dyn}$ is isomorphic to the natural numbers $(\natural,+,0)$.
\end{example}

\subsection{The Category of \texorpdfstring{$\C$}{C}-sets} \label{subsection:category-c-sets}

Recall that for two categories $\C$ and $\D$, the functor category $\D^{\C}$ is the category whose objects are functors $F \maps \C \to \D$ and whose morphisms are natural transformations $\alpha \maps F \tto G$. The category $\cSetC{\C}$ is simply $\setC^{\C}$ (or $\finSetC^{\C}$ when we restrict our attention to finite $\C$-sets). This category enjoys excellent mathematical properties as a result of being a functor category.

\begin{proposition} \label{prop:pointwise-limits}
  Limits and colimits in $\D^{\C}$ are computed pointwise in $\D$. More precisely, if $J$ is a small category and $\D$ has all (co)limits of shape $J$, then $\D^{\C}$ has all (co)limits of shape $J$. Moreover, the limit of a diagram $K \maps J \to \D^{\C}$ is computed via the formula
  \begin{equation*}
    \left(\lim_{j \in J} K(j)\right)(c) = \lim_{j \in J} (K(j)(c))
  \end{equation*}
  for all $c \in \ob \C$, and similarly for colimits.
\end{proposition}

The proof of this standard result can be found in \cite[Sections 8.5-8.6]{awodey_category_2010} or \cite[Proposition 3.3.9]{riehl_category_2016}. Because $\setC$ has all (small) limits and colimits, and $\finSetC$ has all finite limits and colimits, the proposition establishes that $\cSetC{\C} = \setC^{\C}$ and $\finSetC^{\C}$ have the corresponding limits and colimits. The computational content of the proposition is illustrated by the following examples.

\begin{example}
  The product $G_1 \times G_2$ of graphs $G_1$ and $G_2$ has edge set
  \[(G_{1} \by G_{2})(E) = G_1(E) \by G_2(E) \]
  and vertex set
  \[(G_{1} \by G_{2})(V) = G_1(V) \by G_2(V). \]
  Both equations follow from the formula in \cref{prop:pointwise-limits}. The edge $(e,e') \in (G_1 \by G_2)(E)$ in the product graph has source vertex $(G_{1}(\src)(e),G_{2}(\src)(e'))$, and similarly for the target.
\end{example}

\begin{example}
  The coproduct of two discrete dynamical systems $D_{1}$ and $D_{2}$ has state space
  \[ (D_{1} \sqcup D_{2})(X) = D_{1}(X) \sqcup D_{2}(X). \]
  The successor map $(D_{1} \sqcup D_{2})(\suc)$ is defined by
  \[
    (D_{1} \sqcup D_{2})(\suc)(x) = \begin{cases}
      D_{1}(\suc)(x) & \qif* x \in D_{1}(X) \\
      D_{2}(\suc)(x) & \qif* x \in D_{2}(X)
    \end{cases}.
  \]
\end{example}

We are interested in limits and colimits because applications of $\C$-sets frequently involve operations that can be expressed using limits or colimits; we will see later that being able to compute pushouts is essential for composing structured cospans. The pointwise formula leads to a generic algorithm for computing limits and colimits in functor categories, which we have implemented for finite $\C$-sets.

\subsection{Queries and Data Migration}

Given two schemas $\C$ and $\D$ along with a functor $f \maps \C \to \D$, the induced \emph{pullback functor} $f^{\ast} \maps \setC^{\D} \to \setC^{\C}$ from $\D$-sets to $\C$-sets is given by precomposition with $f$.

\begin{example} \label{ex:vertex_map}
  Let $\terminalC = \{\ast\}$ be the terminal category and let $f \maps \terminalC \to \cat{Gr}$ be the functor sending $\ast$ to $V$. Then $f^{\ast} \maps \setC^{\cat{Gr}} \to \setC^{\terminalC} \cong \setC$ is the forgetful functor sending a graph to its set of vertices.
\end{example}

\begin{example}
  Let $\iota$ be the inclusion functor from the category $\cat{Gr}$ into $\cat{SymGr}$. The pullback functor $\iota^{\ast}$ sends a symmetric graph to its underlying graph, forgetting the involution on edges.
\end{example}

For any functor $f$, the pullback functor $f^*$ has left and right adjoints, denoted $\Sigma_{f}$ and $\Pi_{f}$.
\[\begin{tikzcd}
	{\setC^{C}} && {\setC^{D}}
	\arrow[""{name=0, anchor=center, inner sep=0}, "{f^\ast}"{description}, from=1-3, to=1-1]
	\arrow[""{name=1, anchor=center, inner sep=0}, "{\Sigma_f}", curve={height=-18pt}, from=1-1, to=1-3]
	\arrow[""{name=2, anchor=center, inner sep=0}, "{\Pi_f}"', curve={height=18pt}, from=1-1, to=1-3]
	\arrow["\dashv"{anchor=center, rotate=-90}, draw=none, from=1, to=0]
	\arrow["\dashv"{anchor=center, rotate=-90}, draw=none, from=0, to=2]
\end{tikzcd}\]
The simplest way to think about these adjunctions is through the natural bijections they define between hom sets. That is,
\begin{align*}
  \Hom(X, f^{\ast}(Y)) &\cong \Hom(\Sigma_{f}(X), Y) \\
  \Hom(f^{\ast}(Y),X) &\cong \Hom(Y,\Pi_{f}(X))
\end{align*}
for all $\C$-sets $X$ and $\D$-sets $Y$.

Returning to \cref{ex:vertex_map}, we can use these isomorphisms between hom sets to see that $\Sigma_{f}(X)$ is the graph with vertex set $X$ and no edges (the \emph{discrete} graph), and $\Pi_{f}(Y)$ is the graph with vertex set $Y$ and precisely one edge between each pair of vertices (the \emph{codiscrete} graph).

\begin{example} \label{ex:reflexive_graph}
  A \emph{reflexive graph} is a graph where every vertex is equipped with a distinguished self-loop; more precisely, it is a $\cat{ReflGr}$-set, where
  \begin{equation*}
    \cat{ReflGr} := \left\langle
    \begin{tikzcd}
      E & V
      \arrow["{\mathrm{src}}", shift left=2, from=1-1, to=1-2]
      \arrow["{\mathrm{tgt}}"', shift right=2, from=1-1, to=1-2]
      \arrow["{\mathrm{refl}}"{description}, from=1-2, to=1-1]
    \end{tikzcd}
    \middle|\
    \begin{gathered}
      \mathrm{refl} \cmp \mathrm{src} = 1_V \\
      \mathrm{refl} \cmp \mathrm{tgt} = 1_V
    \end{gathered}
    \right\rangle.
  \end{equation*}
  Starting from the evident inclusion $\iota \maps \cat{Gr} \to \cat{ReflGr}$, the functor $\Sigma_{\iota}$ freely adds reflexive edges to each vertex of a graph. This transformation is useful because the product of two line graphs does not produce a mesh (\cref{fig:nonreflexive_graph_product}), but the product of two reflexive line graphs does (\cref{fig:reflexive_graph_product}).
\end{example}

\begin{figure}[ht]
  \begin{subfigure}{\textwidth}
    \centering
    \includegraphics[scale=0.8]{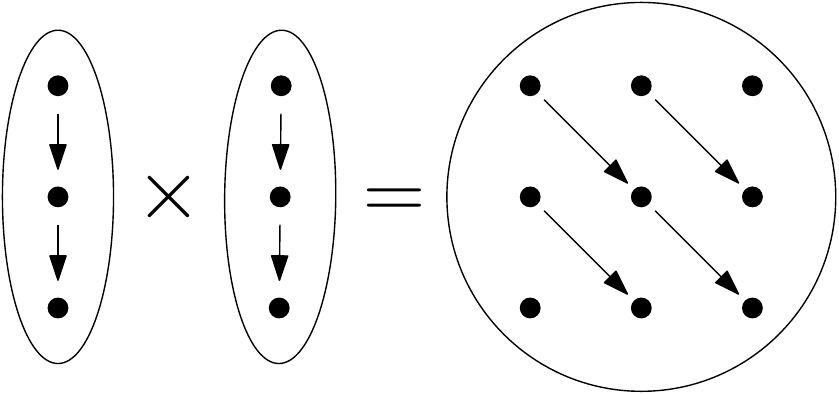}
    \caption{Product of two non-reflexive graphs}
    \label{fig:nonreflexive_graph_product}
  \end{subfigure}
  \begin{subfigure}{\textwidth}
    \centering
    \includegraphics[scale=0.8]{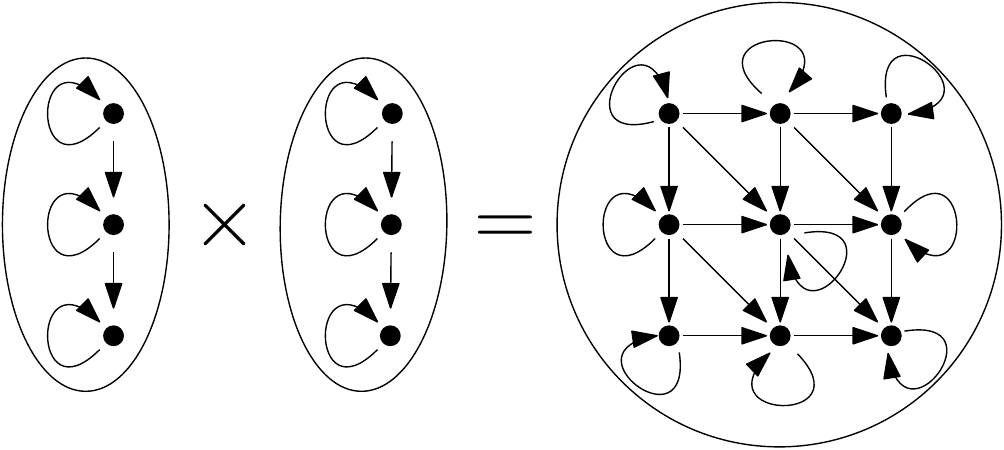}
    \caption{Product of two reflexive graphs}
    \label{fig:reflexive_graph_product}
  \end{subfigure}
  \caption{Categorical products of graphs}
  \label{fig:graph_products}
\end{figure}

For any functor $f$ between finitely presented categories, the pullback functor $f^{\ast}$ is straightforwardly implemented. Computing the pushforward functors $\Sigma_{f}$ and $\Pi_{f}$ can, in principle, be reduced to computing limits and colimits of sets, but the reduction itself can be difficult to compute \cite{bush2003} and the resulting sets can be infinite.

This concludes our brief review of $\C$-sets. $\C$-sets, along with the dual concept of \emph{presheaves} (functors $\C^{\op} \to \setC$), are among the best-studied notions of category theory. A more thorough treatment may be found in textbooks such as \cite{awodey_category_2010,reyes_generic_2004,riehl_category_2016}.

\section{Theory of Attributed \texorpdfstring{$\C$}{C}-sets} \label{section:theory}

In \cref{section:practice}, we contrasted two kinds of information that can be contained in a dataset. The first kind, combinatorial data, is modeled well by $\C$-sets. However, the second kind, attribute data, is not modeled appropriately because the isomorphism class of a $\C$-set abstracts away from the actual elements of the sets. The only information it retains is how the elements are related to each other.

One approach to augmenting $\C$-sets with attribute data, elaborated in \cite{spivak_functorial_2012}, is to consider the slice category $\setC^{\C}/D$ over a fixed $\C$-set $D$. An object of $\setC^{\C}/D$ is a functor $F \maps \C \to \setC$ together with a natural transformation $\alpha \maps F \to D$, whose components $\alpha_c: F(c) \to D(c)$ are assignments of data to the elements of $F(c)$. For instance, data frames can be modeled in this framework. Consider a data frame with column types $X_{1}, \ldots, X_{n}$, which are simply sets. Perhaps $X_{1} = \integer$ and $X_{2} = \real$. Then let $\terminalC = \set{\ast}$ be the terminal category and define the $\terminalC$-set $D$ by $D(\ast) := \prod_{i=1}^{n}X_{i}$. Of course, $\setC^{\terminalC} \cong \setC$, so we have $\setC^{\terminalC}/D \cong \setC/\prod_{i=1}^{n} X_i$. An object of this slice category is a set $U$ together with a function $f \maps U \to \prod_{i=1}^{n} X_{i}$. We interpret it as a data frame by saying that each element $x \in U$ is a row whose $i$th column value is $(\pi_{i} \circ f)(x)$, where $\pi_{i}$ is the $i$th projection function. Note that the morphisms in this category must preserve column values, so an isomorphism class of data frames \emph{does} preserve the attribute data, not just the number of rows.

However, for a general category $\C$ and choice of data attributes, the construction of the $\C$-set $D$ becomes complicated. It is not as simple as sending each object to the product of the attribute types for that object, because morphisms in $\C$ effectively induce extra attributes on their domain objects. For this reason we turn to another definition of {\acset}s, which ultimately reduces to slice categories but is more convenient to work with.

The idea is that we will identify $\C$ with one half of a larger category $\abs{S}$, where the rationale for the notation $\abs{S}$ will become clear shortly. The half of $\abs{S}$ identified with $\C$ specifies the structure of the combinatorial data; the other half specifies the structure of the attribute data. As an example, the schema for decorated graphs is pictured in \cref{fig:dec_graph_schema}. By ``decorated graph,'' we mean a graph whose vertices and edges are decorated with data of a specific type.

\begin{figure}[ht]
  \centering
  \includegraphics[width=0.75\textwidth]{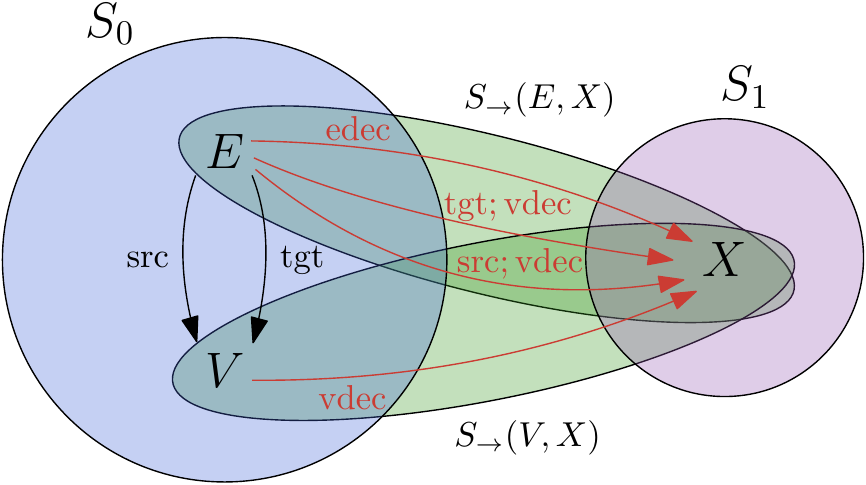}
  \caption{\label{fig:dec_graph_schema} Schema for decorated graphs}
\end{figure}

More formally, we make the following definition.

\begin{definition}
  A \emph{schema} is a small category $\abs{S}$ with a map $S \maps \abs{S} \to \intervalC$, where $\intervalC$ is the interval category $\{0 \to 1\}$, Unless otherwise stated, we assume that $S^{-1}(1)$ is a discrete category.
\end{definition}

We denote the preimages $S^{-1}(0)$ and $S^{-1}(1)$ by $S_{0}$ and $S_{1}$, respectively, and for $c \in S_{0}$ and $X \in S_{1}$, we write $S_{\sto}(c,X) := \Hom_{\abs{S}}(c,X)$. As in the setting of $\C$-sets, we typically require that $\abs{S}$  be finitely presented.

Observe that $S_{\sto}$ defines a functor $S_{0}\op \by S_{1} \to \setC$. This gives an alternative definition of a schema as two categories $S_{0}$ and $S_{1}$ together with a functor $S_{\sto} \maps S_{0}\op \by S_{1} \to \setC$. Such a functor is called a \emph{profunctor} and is denoted by $S_{\sto} \maps S_{0} \pto{} S_{1}$. Loosely speaking, a profunctor is a kind of ``bipartite category.'' The category $\abs{S}$ in the original definition is recovered as the \emph{collage} of the profunctor.

Without further ado, we can define an attributed $\C$-set.

\begin{definition}[Main Definition]
  An \emph{\acset} on a schema $S$ with \emph{typing map} $K \maps S_{1} \to \setC$ is a functor $G \maps \abs{S} \to \setC$ that restricts to $G|_{S_{1}} = K$. It is \emph{finite} if $G$ restricts to a functor $G|_{S_{0}}$ into $\finSetC$.

  A \emph{morphism of {\acset}s} $G_{1}, G_{2} \maps \abs{S} \to \setC$ with typing map $K$ is a natural transformation $\alpha \maps G_{1} \to G_{2}$ such that $\alpha |_{S_{1}}$ is the identity transformation on $K$. The category of $K$-typed {\acset}s on the schema $S$ and morphisms between them is denoted $\acsetC_{K}^{S}$.
\end{definition}

For readers familiar with the double category of profunctors, $\Prof$, an equivalent but perhaps more elegant definition of {\acset}s and {\acset} morphisms is as follows.

\begin{definition}[Alternative Definition] \label{def:acset_profunctor}
  An \emph{\acset} on a schema $S$ with typing $K \maps S_{1} \to \setC$ is a functor $F \maps \C \to \setC$ together with a 2-cell $\alpha$ in $\Prof$ of form
  \[
    \begin{tikzcd}
      S_{0} \ar[r,tick,"S_{\sto}",""{name=top,inner sep=5pt,below}] \ar[d,swap,"F"] & S_{1} \ar[d,"K"] \\
      \setC \ar[r,swap,tick,"\Hom_{\setC}",""{name=bottom,inner sep=5pt,above}] & \setC
      \ar[from=top,to=bottom,Rightarrow,"\alpha"]
    \end{tikzcd}.
  \]

  A \emph{morphism of {\acset}s} $(F,\alpha)$ and $(G,\beta)$ on a schema $S$ with typing $K$ is a natural transformation $\gamma \maps F \tto G$ such that
  \[
    \begin{tikzcd}
      S_{0} \ar[r,tick,"\Hom_{S_{0}}",""{name=top1,inner sep=5pt,below}] \ar[d,swap,"F"] & S_{0} \ar[r,tick,"S_{\sto}",""{name=top2,inner sep=5pt,below}] \ar[d,swap,"G"] & S_{1} \ar[d,"K"] \\
      \setC \ar[r,swap,tick,"\Hom_{\setC}",""{name=bottom1,inner sep=5pt,above}] & \setC \ar[r,swap,tick,"\Hom_{\setC}",""{name=bottom2,inner sep=5pt,above}] & \setC
      \ar[from=top1,to=bottom1,Rightarrow,"\gamma"] \ar[from=top2,to=bottom2,Rightarrow,"\beta"]
    \end{tikzcd} = \begin{tikzcd}
      S_{0} \ar[r,tick,"S_{\sto}",""{name=top,inner sep=5pt,below}] \ar[d,swap,"F"] & S_{1} \ar[d,"K"] \\
      \setC \ar[r,swap,tick,"\Hom_{\setC}",""{name=bottom,inner sep=5pt,above}] & \setC
      \ar[from=top,to=bottom,Rightarrow,"\alpha"]
    \end{tikzcd}.
  \]
\end{definition}

Our definition of an {\acset} is closely related to other definitions of categorical databases with attributes in the literature \cite{spivak_wisnesky_2015,schultz_algebraic_2016}. \cref{def:acset_profunctor} is a simpler but less expressive variant of the ``algebraic databases'' introduced by Schultz et al \cite{schultz_algebraic_2016}, replacing \emph{algebraic profunctors}---product-preserving profunctors into multisorted algebraic theories---with profunctors into discrete categories. This means that our data model excludes operations on data attributes, although of course such operations can still be performed in ordinary Julia code. On the other hand, our definition is slightly richer than that of Spivak and Wisnesky \cite{spivak_wisnesky_2015}, including a notion of \emph{attribute type} rather than treating each attribute as having its own independent data type. This addition is convenient in our implementation, as attribute types correspond to type parameters in the generated Julia data type (see \cref{section:implementation}). In summary, our notion of {\acset} is intermediate in complexity between those in \cite{spivak_wisnesky_2015} and \cite{schultz_algebraic_2016}.

\subsection{\texorpdfstring{$\acsetC_{K}^{S}$}{Acsets} as a Slice Category}

In this section, we show that the category $\acsetC_{K}^{S}$ is a slice category in $\setC^{\C}$, where $\C = S_0$. As a result, $\acsetC_{K}^{S}$ inherits useful categorical properties from $\setC^{\C}$, such as completeness and cocompleteness.

\begin{theorem} \label{thm:main_theorem}
  The category $\acsetC_{K}^{S}$ is isomorphic to the slice category $\setC^{\C}/D$ for some $\C$-set $D$.
\end{theorem}

\begin{proof}
  The proof happens in two major steps. The first step is to establish that acsets are fully captured by the diagram in \cref{fig:acset_nat_trans}; the second is to recognize the possibility of constructing $D$ from a certain Kan extension.

  \begin{figure}
    \centering
    \includegraphics{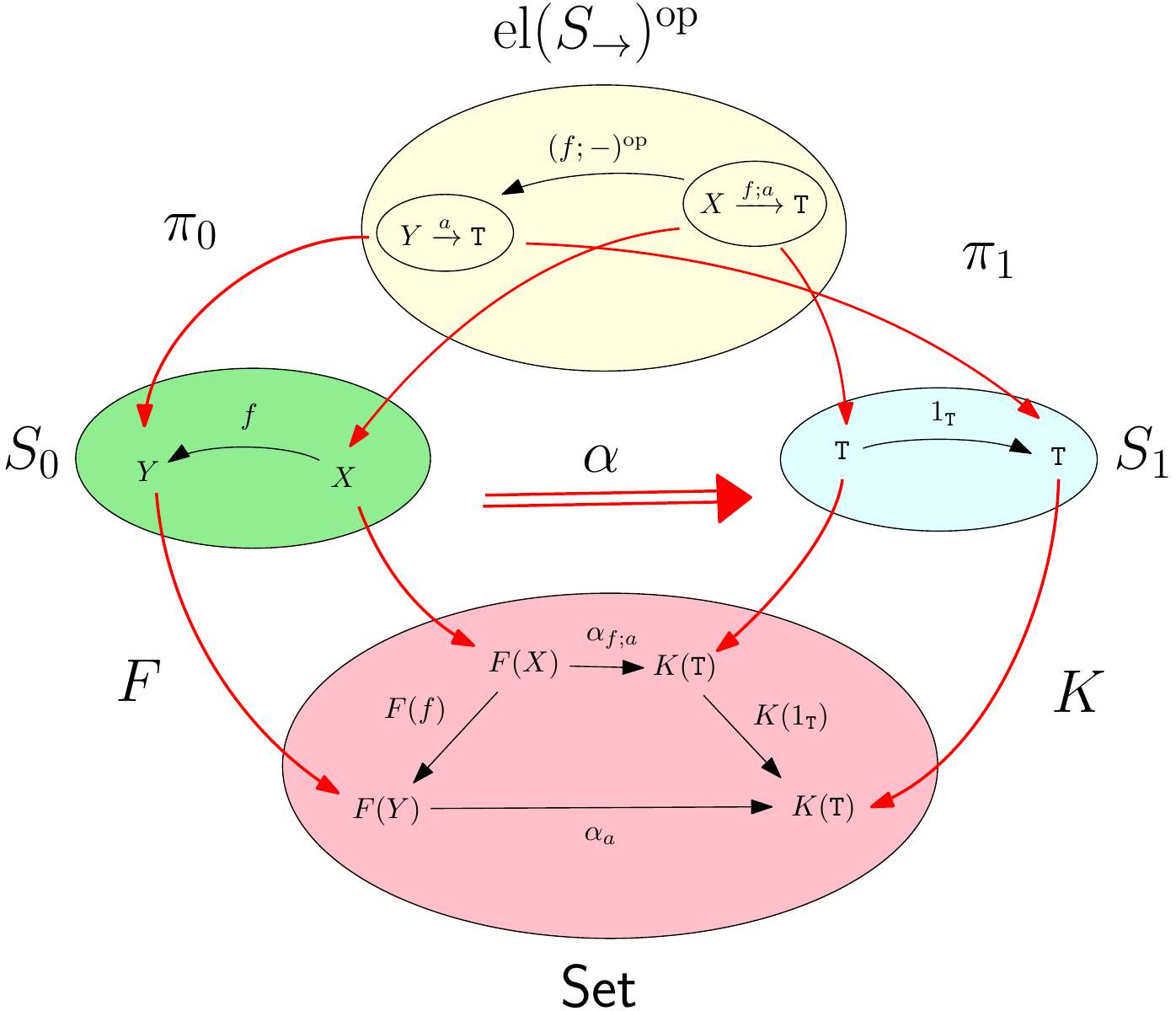}
    \caption{An {\acset} as a natural transformation $\alpha$ from $\pi_0 \cmp F$ to $\pi_1 \cmp K$}
    \label{fig:acset_nat_trans}
  \end{figure}

  Recall the alternative definition of the schema $S$ as a functor $S_{\sto} \maps S_{0}\op \by S_{1} \to \setC$. As $S_{1}$ is discrete, $S_{1} = S_{1}\op$ and we can reinterpret $S_{\sto}$ as a functor $S_{\sto} \maps (S_{0} \by S_{1})\op \to \setC$. Let $\el(S_{\sto})$ be the category of elements of $S_{\sto}$, whose objects are attributes $Y \xrightarrow{a} \mathtt{T}$, and let be $p \maps \el(S_{\sto}) \to (S_{0} \by S_{1})\op$ the canonical projection, sending $Y \xrightarrow{a} \mathtt{T}$ to the pair $(Y,\mathtt{T})$. Then for any {\acset} $G$ in $\acsetC_{K}^{S}$, consisting of a functor $G: \abs{S} \to \setC$ with $G|_{S_{0}} =: F$ and $G|_{S_{1}} = K$, we have a diagram
\[
  \begin{tikzcd}[column sep=small]
    & \el(S_{\sto})\op \ar[ld,swap,"\pi_{0} \circ p"] \ar[rd,"\pi_{1} \circ p"]\\
    S_{0} \ar[rd,swap,"F"] \ar[rr,shorten <=20pt,shorten >=20pt,Rightarrow,"\alpha"] & & S_{1} \ar[ld,"K"]\\
    & \setC
  \end{tikzcd}.
\]
The natural transformation $\alpha$ is defined by letting its component at an element $a \in \el(S_{\sto})\op$, which is an attribute from $(\pi_{0} \circ p)(a) \in S_{0}$ to $(\pi_{1} \circ p)(a) \in S_{1}$, be the function
\[
\alpha_a := G(a) \maps F((\pi_{0} \circ p)(a)) \to K((\pi_{1} \circ p)(a)).
\]
This construction is depicted schematically in \cref{fig:acset_nat_trans}. The naturality of $\alpha$ follows from the functorality of $G$. That is, $\alpha_{f;a} = F(f) ; \alpha_{a}$ because $G(f;a) = G(f) ; G(a) = F(f) ; G(a)$. Moreover, any such natural transformation $\alpha$ defines a functor $G: \abs{S} \to \setC$ with $G|_{S_{1}} = K$. This correspondence gives an isomorphism between the category of {\acset}s and the category of pairs $(F,\alpha)$. In the latter category, the morphisms $(F_{1},\alpha_{1}) \to (F_2,\alpha_{2})$ are natural transformations $F_{1} \to F_{2}$ making a certain diagram commute, so that attributes are preserved. For further intuition, see \cref{def:acset_profunctor} of the category of {\acset}s in terms of the double category of profunctors.

Next, we rewrite the above diagram in a slightly different form:
\[
  \begin{tikzcd}
    \el(S_{\sto})\op \ar[rr,"K \circ \pi_{1} \circ p"] \ar[rd,swap,"\pi_{0} \circ p"] & {} & \setC \\
    & S_0 \ar[ur,swap,"F"] \ar[u,shorten <=5pt, Rightarrow,"\alpha"]
  \end{tikzcd}
\]
Diagrams of this form are precisely the context for right Kan extensions. Letting $\tilde{p} := \pi_{0} \circ p$ and $\tilde{K} := K \circ \pi_{1} \circ p$, the right Kan extension of $\tilde K$ along $\tilde p$, denoted $\Ran_{\tilde{p}} \tilde K$, is known to exist \cite[Corollary 6.2.6]{riehl_category_2016}. By its defining universal property, the natural transformations $\alpha$ in the above diagram are in isomorphism with natural transformations $\hat{\alpha}$ in the diagram below:
\[
  \begin{tikzcd}[row sep=large]
    \el(S_{\sto})^{\op} \ar[rr,"\tilde{K}"] \ar[rd,swap,"\tilde{p}"] & {} & \setC \\
    & S_{0} \ar[ur,bend right=25,swap,"F",""{name=F,inner sep=0.5pt,above}] \ar[ur,bend left=25,"\Ran_{\tilde{p}}\tilde{K}",,""{name=lan,inner sep=0.5pt,below}]
    \ar[Rightarrow,from=F,to=lan,swap,"\hat{\alpha}"]
  \end{tikzcd}
\]

Let $D = \Ran_{\tilde{p}}\tilde{K}$. We have shown that an object of $\acsetC_{K}^{S}$ is precisely a functor $F \maps S_{0} \to \setC$ together with a natural transformation $\hat\alpha \maps F \tto D$. Moreover, tracing carefully through the argument shows that a morphism of {\acset}s is exactly a morphism in the slice category $\setC^{S_{0}}/D$, which completes the proof.
\end{proof}

A similar result is stated by Spivak and Wisnesky \cite[Proposition 9.1.2]{spivak_wisnesky_2015}, although the proof there is more immediate due to differences in formalism.

In the very special case where $\cat{C} = \{*\}$ is the terminal category, the $\cat{C}$-set $D$ in \cref{thm:main_theorem}, which is just a set, is the product $\prod_{* \xrightarrow{a} \mathtt{T}} K(\mathtt{T})$ of the attribute types of all attributes in the schema. This is the example of a data frame discussed at the beginning of the section. In general, however, the right Kan extension $D = \Ran_{\tilde p} \tilde K$ used in the proof is large, complicated, and impractical to instantiate in software. Thus, it should be understood mainly as a theoretical device for understanding the category of {\acset}s.

Indeed, it is useful to know that the category of acsets is a slice category of a presheaf category, as such categories have many desirable features. Because presheaf categories are elementary toposes and slices of toposes are again toposes (the ``fundamental theorem of toposes'' \cite[Theorem 17.4]{mclarty1992}), the category of acsets is a topos. In particular, all finite limits and colimits exist, and can be computed by known formulas. Also, there is a geometric morphism between $\setC^{S_{0}}$ and the category of acsets on $S$.  Thus, in an abstract sense, the important properties of the category of acsets are determined by the above construction and already known; the main innovation is in realizing the many useful applications thus enabled. Much of \cref{section:implementation} will be dedicated to showing how the properties of slice categories of presheaf categories translate into practical capabilities for designing and implementing software.

\subsection{Functorality of the Attributed $\C$-set Construction}

Careful examination of the construction of $\acsetC^{S}_{K}$ shows that it is functorial in both $S$ and $K$.

Functorality in $S$ is given simply by precomposition. A \emph{morphism of schemas} $S \to S'$ is a functor $L \maps \abs{S} \to \abs{S'}$ that restricts to functors $L_{0}: S_{0} \to S_{0}'$ and $L_{1}: S_{1} \to S_{1}'$. Given such a functor $L$ and a typing map $K \maps S'_{1} \to \setC$, there is a functor $\Delta_{L} \maps \acsetC_{K}^{S'} \to \acsetC_{L_{1} \cmp K}^{S}$ which sends $G \maps S' \to \setC$ to $L \cmp G \maps S \to \setC$.

\begin{example}
  Recall the schema for weighted graphs. For this schema, $S_{1} \cong \set{\ast}$ has a single object, so a typing map is simply a choice of a set. Functorality of the {\acset} construction in the typing map allows us to lift a map $\real \to \complex$ to a map from $\real$-weighted graphs to $\complex$-weighted graphs.

  Moreover, the inclusion of the schema for ``weighted sets'' into the schema for weighted graphs, pictured below, induces a morphism from weighted graphs to weighted sets that sends a weighted graph to its weighted set of edges.
  \[\begin{tikzcd}
      \bullet && E & V \\
      {\mathtt{T}} && {\mathtt{T}}
      \arrow[from=1-1, to=2-1]
      \arrow["{\mathrm{tgt}}"', curve={height=6pt}, from=1-3, to=1-4]
      \arrow["{\mathrm{src}}", curve={height=-6pt}, from=1-3, to=1-4]
      \arrow[from=1-3, to=2-3]
      \arrow[curve={height=-12pt}, dotted, maps to, from=1-1, to=1-3]
      \arrow[curve={height=-12pt}, dotted, maps to, from=2-1, to=2-3]
    \end{tikzcd}\]
\end{example}

For functorality in $K$, suppose that $\gamma \maps K \tto L$ is a natural transformation. Then the Kan extension in the proof of \cref{thm:main_theorem} yields a map $\hat{\gamma} \maps P := \Ran_{\pi_{0} \cmp F} (\pi_{1} \cmp K) \to \Ran_{\pi_{0} \cmp F} (\pi_{1} \cmp L) =: Q$. The theory of slices of toposes \cite[Chapter 11]{mclarty1992} then produces three functors:
\[\begin{tikzcd}
	{\setC^{S_0}/P} && {\setC^{S_0}/Q}
	\arrow[""{name=0, anchor=center, inner sep=0}, "{\hat{\Sigma}_{\hat{\gamma}}}", curve={height=-18pt}, from=1-1, to=1-3]
	\arrow[""{name=1, anchor=center, inner sep=0}, "{\hat{\Pi}_{\hat{\gamma}}}"', curve={height=18pt}, from=1-1, to=1-3]
	\arrow[""{name=2, anchor=center, inner sep=0}, "{\hat{\Delta}_{\hat{\gamma}}}"{description}, from=1-3, to=1-1]
	\arrow["\dashv"{anchor=center, rotate=-90}, draw=none, from=0, to=2]
	\arrow["\dashv"{anchor=center, rotate=-90}, draw=none, from=2, to=1]
\end{tikzcd}\]

It is important to note that these three functors are \emph{different} from the three functors given by functorality in $S_{0}$. Namely, these functors do not migrate {\acset}s between different schemas, and instead manipulate the data of an {\acset}.

\begin{example}
  In the case that the schema has only one object, no morphisms, and one attribute, these functors are particularly simple. Let $P$ and $Q$ be sets, and let $\gamma \maps P \to Q$ be any function. We then have adjunctions:
  \[\begin{tikzcd}
      {\setC/P} && {\setC/Q}
      \arrow[""{name=0, anchor=center, inner sep=0}, "{\hat{\Sigma}_{\hat{\gamma}}}", curve={height=-18pt}, from=1-1, to=1-3]
      \arrow[""{name=1, anchor=center, inner sep=0}, "{\hat{\Pi}_{\hat{\gamma}}}"', curve={height=18pt}, from=1-1, to=1-3]
      \arrow[""{name=2, anchor=center, inner sep=0}, "{\hat{\Delta}_{\hat{\gamma}}}"{description}, from=1-3, to=1-1]
      \arrow["\dashv"{anchor=center, rotate=-90}, draw=none, from=0, to=2]
      \arrow["\dashv"{anchor=center, rotate=-90}, draw=none, from=2, to=1]
  \end{tikzcd}\]
  The functors $\hat{\Sigma}_{\hat{\gamma}}$ and $\hat{\Delta}_{\hat{\gamma}}$ have simple descriptions: $\hat{\Sigma}_{\hat{\gamma}}$ composes $X \to[f] P$ with $\gamma$, and $\hat{\Delta}_{\hat{\gamma}}$ pulls back $Y \to[g] Q$ along $\gamma$. To describe $\hat{\Pi}_{\hat{\gamma}}$, it is best to exploit the isomorphism $\setC/P \cong \setC^{P}$, viewing $P$ as a discrete category. For more information about adjoint triples arising from slices, a standard reference is \cite[Chapter 1]{maclane1994}.
\end{example}

\begin{example}
  Consider again weighted graphs. If $G$ is a $\integer$-weighted graph and $f \maps \integer \to \real$ is any function, then the functor $\hat{\Sigma}_{\hat{f}}$ ``maps $f$ over $G$'' to produce a $\real$-weighted graph with the same underlying graph as $G$.

  The other two functors are not as straightforward. When $\gamma$ is injective, the functor $\hat{\Delta}_{\hat{\gamma}}$ can be interpreted as a ``filtering transformation,'' which deletes the parts of the {\acset} having attributes that are not in the domain of $\gamma$. For example, if $f \maps [0,200] \hookrightarrow \real_{\geq 0}$ is the inclusion, then $\hat{\Delta}_{\hat{f}}$ takes a $\real_{\geq 0}$-weighted graph and deletes any edges with weights greater than 200 to produce a $[0,200]$-weighted graph. More generally, when $\gamma$ is not injective, parts can be copied as well as deleted. For the unique function $f \maps \set{0,1} \to \set{\ast}$, the functor $\hat{\Delta}_{\hat{f}}$ maps a graph with weights of singleton type (i.e., an unweighted graph) to the graph where each edge is duplicated, one copy weighted by $0$ and the other by $1$.
\end{example}

In general, $\hat{\Delta}_{\hat{\gamma}}$ performs a combination of these two transformations. Again, $\hat{\Pi}_{\hat{\gamma}}$ is not so intuitive, especially in the context of {\acset}s. The reader is encouraged to try examples for themselves and use the adjointness property to see what $\hat{\Pi}_{\hat{\gamma}}$ must be.

\subsection{Structured Cospans with Attributed $\C$-sets} \label{subsection:structured_cospans}

{\Acset}s are usefully combined with structured cospans, a formalism for open systems introduced by Baez and Courser \cite{baez_structured_2020}, building on Fong's decorated cospans \cite{fong_decorated_2015}. Categories of structured cospans generalize categories of cospans, which we now review.

\begin{definition}
  Given a category $\C$ with pushouts, the \emph{cospan category} $\cospan(\C)$ has the same objects as $\C$ and has as morphisms from $a$ to $b$, the isomorphism classes of cospans $a \to x \ot b$ in $\C$. Morphisms $a \to x \ot b$ and $b \to y \ot c$ are composed by pushout:
  \begin{equation*}
  \begin{tikzcd}
  	&& {x +_b y} \\
  	& x && y \\
  	a && b && c
  	\arrow[from=3-1, to=2-2]
  	\arrow[from=3-3, to=2-2]
  	\arrow[from=3-3, to=2-4]
  	\arrow[from=3-5, to=2-4]
  	\arrow[from=2-2, to=1-3]
  	\arrow[from=2-4, to=1-3]
  	\arrow["\ulcorner"{anchor=center, pos=0.125, rotate=-45}, draw=none, from=1-3, to=3-3]
  \end{tikzcd}
  \end{equation*}
  The identity morphism for $c \in \C$ is the cospan $c \to[1_c] c \ot[1_c] c$.
\end{definition}

The definition involves one technical point, which is that a pushout, like any colimit, is defined only up to canonical isomorphism. The most direct way to solve this problem is to work with isomorphism classes of cospans, rather than with cospans themselves, as done above. Two cospans $a \to[f] c \ot[g] b$ and $a \to[f'] c' \ot[g'] b$ are \emph{isomorphic} if there exists an isomorphism $h \maps c \to c'$ in $\C$ making the following diagram commute:
\[\begin{tikzcd}
	& c \\
	a && b \\
	& {c'}
	\arrow["f", from=2-1, to=1-2]
	\arrow["g"', from=2-3, to=1-2]
	\arrow["{g'}", from=2-3, to=3-2]
	\arrow["{f'}"', from=2-1, to=3-2]
	\arrow["h"{description}, from=1-2, to=3-2]
\end{tikzcd}.\]

Structured cospans extend the concept of a cospan to the setting where the apex of the cospan has more structure than the feet. Thus, the feet of the cospan will come from a category $\C$ while the apex will belong to a different category $\D$, which is related to $\C$ by a functor $L \maps \C \to \cat{D}$.

\begin{definition}
  Given categories $\C$ and $\cat{D}$, where $\cat{D}$ has all pushouts, along with a functor $L \maps \C \to \cat{D}$, the \emph{$L$-structured cospan category} $_{L}\cospan(\cat{D})$ has the same objects as $\C$ and has as morphisms from $a$ to $b$, the isomorphism classes of cospans $L(a) \to x \ot L(b)$, where $a,b \in \C$ and $x \in \cat{D}$. As before, morphisms are composed by pushout.
\end{definition}

Most combinatorial examples of structured cospans, including many examples in
\cite{baez_structured_2020}, can be constructed using categories of {\acset}s
and data migration functors. A common pattern, illustrated by the following
examples, takes the category $\cat{D}$ to be an arbitrary category of acsets on
a schema $S$ and takes $\cat{C}$ to be the category of acsets on a schema
comprising a chosen object $c$ of $S$ along with any data attributes of $c$ in $S$, under the assumption that $c$ has no outgoing morphisms in $S$. The functor $L: \cat{C} \to \cat{D}$ is then the left adjoint to the forgetful functor $\cat{D} \to \cat{C}$.

\begin{example}
  Consider the inclusion of the terminal category $\cat{1}$ into the schema for weighted graphs that sends the unique object of $\cat{1}$ to $V$. Pullback along this inclusion is the forgetful functor that sends a weighted graph to its set of vertices, and the left adjoint $L$ to this forgetful functor sends a set to the \emph{discrete weighted graph} having that set as vertices and no edges. The $L$-structured cospans are open weighted graphs, which are composed by gluing together vertices.
\end{example}

\begin{example}
  Returning to the whole-grained Petri nets of \cref{ex:petri_nets}, consider the inclusion of the terminal category $\cat{1}$ into the ``SITOS schema'' (\cref{fig:petri_net_vanilla}) that sends the unique object of $\cat{1}$ to $S$. Pullback along this inclusion sends a Petri net to its set of species, and the left adjoint sends a set to the Petri net having that set as species and no transitions. Structured cospans with respect to this functor are open, whole-grained Petri nets. Structured cospans of Petri nets with rates, where each transition has an associated rate constant, can be handled similarly.
\end{example}

Because left adjoints preserve colimits, the conditions needed to define categories of structured cospans are automatically satisfied. The fact that so many common examples of structured cospans are in fact {\acset}s with data migration functors enables a uniform implementation of these examples.

\subsection{Limits and Colimits} \label{subsection:limits_and_colimits}

In this section we give explicit formulas for computing limits and colimits in categories of {\acset}s. Since, by the main theorem \ref{thm:main_theorem}, the category $\acsetC_{K}^{S}$ is a slice category in the category of $\C$-sets, the limit and colimit formulas may be derived from the corresponding formulas for $\C$-sets, reviewed in \cref{subsection:category-c-sets}, and for slice categories, reviewed here.

\begin{proposition}
  For any category $\D$ and object $d \in \D$, the forgetful functor $\Pi: \D/d \to \D$ creates colimits in the slice category $\D/d$. In particular, if $\D$ has colimits of shape $\cat{J}$, then for any diagram $F \maps \cat{J} \to \D/d$, we have
  $\Pi(\colim_{\cat{J}} F) \cong \colim_{\cat{J}} F \cmp \Pi$.
\end{proposition}

For a proof (of the dual statement), see \cite[Proposition 3.3.8]{riehl_category_2016}. In conjunction with the colimit formula for $\C$-sets, we deduce:

\begin{corollary}
  The category $\acsetC_{K}^{S}$ has all finite colimits, which are computed pointwise. To be precise, if $F \maps J \to \acsetC_{K}^{S}$ is a finite diagram of {\acset}s on the schema $S$, then for each object $c \in S_{0}$, we have $\evalat_{c}(\colim_{J} F) = \colim_{J} F \cmp \evalat_{c}$, where the evaluation functor $\evalat_{c} \maps \acsetC_{K}^{S} \to \setC$ gives the set associated with the object $c$.
\end{corollary}

Limits in slice categories are less straightforward: they can be reduced to limits in the base category, but only by enlarging the diagram. Given a diagram $F: \cat{J} \to \D/d$ in a slice category, construct a diagram $\bar F: \cat{J}_\bullet \to \D$ in the base category as follows. Let $\cat{J}_\bullet$ be the category given by freely adjoining a terminal object to $\cat{J}$, i.e., a new object $\bullet$ and for each object $j$ in $\cat{J}$, a unique morphism $j_\bullet: j \to \bullet$. Then define $\bar F: J_\bullet \to \D$ by setting $\bar F(\bullet) = d$, $\bar F(j_\bullet) = F(j)$ for each $j \in \cat{J}$, and $\bar F|_{\cat{J}} = F \cmp \Pi$, where $\Pi: \D/d \to \D$ is the forgetful functor. The following result appears in the literature as, for example, \cite[Theorem 17.2]{wyler1991}.

\begin{proposition}
  A diagram $F: \cat{J} \to \D/d$ has a limit in the slice category $\D/d$ if and only if the enlarged diagram $\bar F: \cat{J}_\bullet \to \D$ has a limit in $\D$, in which case the limit cone in $\D$ uniquely lifts to a limit cone in $\D/d$.
\end{proposition}

\begin{corollary}
  The category $\acsetC^{S}_{K}$ has all finite limits, which may be reduced to limits in $\setC^\C$ by the above procedure.
\end{corollary}

For example, the product of two {\acset}s $(F_{1},\beta_{1})$ and $(F_{2}, \beta_{2})$ in $\setC^{\C}/D \cong \acsetC^{S}_{K}$ is computed as a pullback in $\setC^{\C}$:
\[
  \begin{tikzcd}
    & F_{1} \by_{D} F_{2} \ar[rd] \ar[ld] \\
    F_{1} \ar[rd,"\beta_{1}",swap] & & F_{2} \ar[ld,"\beta_{2}"] \\
    & D
  \end{tikzcd}.
\]

\section{Implementation of Attributed \texorpdfstring{$\C$}{C}-sets} \label{section:implementation}

In this section we sketch the more interesting aspects of the implementation of attributed $\C$-sets, especially those that take advantage of unique features of Julia or Catlab.

\subsection{The \texorpdfstring{\jul{AttributedCSet}}{AttributedCSet} Data Structure}

At the core of Catlab is a system to specify and manipulate generalized algebraic theories (GATs) and presentations of their models \cite{cartmell_generalised_1986}, which uses Julia's LISP-style metaprogramming capabilities. Using the \jul{@theory} macro, we can define the theory of a schema to be the theory of a category extended with two new GAT types, for attribute data (\jul{AttrType}) and data attributes (\jul{Attr}). We also declare a syntax system for schemas using the \jul{@syntax} macro, which generates new Julia types to be used in symbolic expressions and in schemas presented by generators and relations.

\begin{minted}{julia}
@theory Schema{Ob,Hom,AttrType,Attr} <: Category{Ob,Hom} begin
  Data::TYPE
  Attr(dom::Ob,codom::AttrType)::TYPE

  compose(f::Hom(A,B), g::Attr(B,X))::Attr(A,X) ⊣ (A::Ob, B::Ob, X::AttrType)

  (compose(f, compose(g, a)) == compose(compose(f, g), a)
    ⊣ (A::Ob, B::Ob, C::Ob, X::AttrType, f::Hom(A,B), g::Hom(B,C), a::Attr(C, X)))
  compose(id(A), a) == a ⊣ (A::Ob, X::Ob, a::Attr(A,X))
end

@syntax FreeSchema{ObExpr,HomExpr,DataExpr,AttrExpr} Schema begin
end
\end{minted}
We can now give finite presentations of schemas using the \jul{@present} macro, as we saw in \cref{section:practice}. For example, the schema for weighted graphs is:
\begin{minted}{julia}
@present TheoryWeightedGraph(FreeSchema) begin
  (V,E)::Ob
  (src,tgt)::Hom(E,V)
  X::AttrType
  dec::Attr(E,X)
end
\end{minted}

The final step in defining an acset is to create the Julia type for the data structure.
\begin{minted}{julia}
@acset_type WeightedGraph(TheoryWeightedGraph, index=[:src,:tgt])
\end{minted}
Inspecting the resulting type \jul{WeightedGraph}, we see that a struct has been generated, suitable for storing the data of a weighted graph.
\begin{minted}{julia}
> dump(WeightedGraph{Int})
WeightedGraph{Int64} <: StructACSet{
    Catlab.Theories.SchemaDescType{
      (:V, :E), (:src, :tgt), (:X,), (:dec,),
      (src = 2, tgt = 2, dec = 2), (src = 1, tgt = 1, dec = 1)
    },
    Tuple{Int64},
    (src = true, tgt = true, dec = false),
    (src = false, tgt = false, dec = false)
  }
  obs::StaticArrays.MVector{2, Int64}
  homs::NamedTuple{(:src, :tgt), Tuple{Vector{Int64}, Vector{Int64}}}
  attrs::NamedTuple{(:dec,), Tuple{Vector{Int64}}}
  hom_indices::NamedTuple{(:src, :tgt), Tuple{Vector{Vector{Int64}}, Vector{Vector{Int64}}}}
  hom_unique_indices::NamedTuple{(), Tuple{}}
  attr_indices::NamedTuple{(), Tuple{}}
  attr_unique_indices::NamedTuple{(), Tuple{}}
\end{minted}
This struct subtypes the abstract type \jul{StructACSet}, with parameters that describe the schema of a WeightedGraph. This allows us to write functions that accept a \jul{StructACSet} and dispatch on the type parameters of \jul{StructACSet} to generate custom code for accessing our weighted graph. We discuss this more in the next section.

\subsection{Code Generation}
\label{subsection:code-generation}

If every low-level operation on an acset had to process the type-level description of the schema, the resulting overhead would make it impossible for acsets to compete with handwritten data structures. Fortunately, it is possible to avoid this runtime penalty using \emph{generated functions}, a useful metaprogramming feature supported by the Julia language. Like a macro, a generated function runs arbitrary Julia code to generate a Julia expression that is then evaluated. However, while a macro operates on the \emph{syntax} of its arguments, a generated function operates on the \emph{types} of its arguments and returns a Julia expression for the body of the function that is specific to those types. Because the type of an {\acset} fully describes its structure, this type contains enough information to generate specialized, fast functions for each operation. After the first run, this specialized code is cached, so that subsequent applications of the function need not generate the code again.

It is best to illustrate the usefulness of generated functions with an example that does not involve {\acset}s, as the types of {\acset}s are rather complicated. Instead, we use the classic example of unrolling the loop in a map operation on vectors of fixed length. In the following code, the type \jul{StaticVector{n,T}} represents a vector of length \jul{n} with element type \jul{T}.
\begin{minted}{julia}
@generated function map(f::Function, v::StaticVector{n,T}) where {n,T}
  if n == 0; error("Type inference for empty vectors not implemented") end
  quote
    SVector{$n}($([:(f(v[$i])) for i in 1:n]...))
  end
end
\end{minted}
When we call the function \jul{map} on an argument of type, say, \jul{StaticVector{3,Int}}, it is as if \jul{map} were defined as:
\begin{minted}{julia}
function map(f::Function, v::StaticVector{3,Int})
  SVector{3}(f(v[1]), f(v[2]), f(v[3]))
end
\end{minted}
In a proper implementation of statically-sized vectors, such as in \texttt{StaticArrays.jl}, much of the performance boost comes from static vectors being allocated on the stack rather than the heap. The loop unrolling would then enable generically sized static vectors to be competitive with code special-cased for particular sizes. Thus, for example, Euclidean geometry packages need not treat the two-dimensional and three-dimensional cases specially for the sake of performance.

Similarly, as the benchmarks below show, generated functions enable a generic implementation of {\acset}s to be competitive with special-cased implementations of specific {\acset}s. Not only does this obviate the need for a great deal of domain-specific libraries, it opens the door to specialized data structures that would previously have required too much effort to be worth writing. In particular, it removes the need for packages to use dictionaries to support arbitrary, user-defined data attributes, such as in \texttt{MetaGraphs.jl}; instead, the user can simply create a new, specialized data structure that has precisely the fields needed for the application at hand.

Almost all basic operations on {\acset}s are implemented as generated functions. This yields a greater savings of programmer effort than it may initially appear due to the support for \emph{indexing}, where both forward and the inverse image maps are stored to enable fast lookups. Operations such as adding and removing elements and changing the values of morphisms interact with the indices in subtle ways. By writing the accessors and mutators using generated functions, this bookkeeping can be handled once and for all.

\subsection{Low-level Operations}

We now demonstrate how to write high-performance algorithms using the low-level interface to {\acset}s. As an example, we implement depth first search on a graph. The following code searches a graph \texttt{g} in a depth-first manner, starting from a vertex \texttt{s}. It returns the array of parent vertices from the search, indexed by vertex.

\begin{minted}{julia}
function dfs_parents(g::Graph, s::Int)
  n = nparts(g, :V) # number of vertices in the graph
  parents = zeros(Int, n)
  seen = zeros(Bool, n)
  S = [s]
  seen[s] = true
  parents[s] = s
  while !isempty(S)
    v = S[end]
    u = 0
    outedges = incident(g, v, :src) # all edges with source vertex v
    outneighbors = subpart(g, outedges, :tgt) # all vertices with edge from v
    for n in outneighbors
      if !seen[n]
        u = n
        break
      end
    end
    if u == 0
      pop!(S)
    else
      seen[u] = true
      push!(S, u)
      parents[u] = v
    end
  end
  return parents
end
\end{minted}

In this code, the call to \texttt{incident} returns the list of edges outgoing from a given vertex \texttt{v}, which uses the index for the morphism \texttt{src} maintained by the {\acset}. The subsequent call to \texttt{subpart} gives the list of vertices with an incoming edge from \texttt{v}. This pattern of data access is repeated in a loop. Relational databases that only provide high-level, query-based access tend to perform poorly on graph algorithms, including many search algorithms, that are highly iterative or recursive \cite{cheng2019benchmarks}.

\subsection{Categorical Operations}

One use for {\acset}s is to supply the ``structure'' in structured cospans, as explained in \cref{subsection:structured_cospans}. To support this application, we must be able to compute pushouts. In this section, we give a brief overview of how pushouts of {\acset}s are computed, in the special case of coequalizers.

Recall from \cref{subsection:limits_and_colimits} that colimits of {\acset}s are computed pointwise. Thus, before we discuss how to compute coequalizers in a category of {\acset}s, we first discuss how to compute a coequalizer in the category of finite sets (cf. \cite[\textsection 4.6]{rydeheard1988}). For the purposes of this discussion, we will use the following data structures for finite sets of form $\{1,\dots,n\}$ and functions between them. Note that this code is significantly simpler and less generic than what is implemented in Catlab.
\begin{minted}{julia}
struct FinSet
  n::Int
end
struct FinFunction
  dom::FinSet
  codom::FinSet
  values::Vector{Int}
  function FinFunction(values::Vector{Int}, codom::Int)
    @assert all(1 <= v <= codom for v in values)
    new(FinSet(length(values)), FinSet(codom), values)
  end
end
\end{minted}
We want to take the coequalizer of a parallel pair of morphisms, given by the following type.
\begin{minted}{julia}
struct ParallelPair
  f::FinFunction
  g::FinFunction
  function ParallelPair(f::FinFunction, g::FinFunction)
    @assert f.dom == g.dom && f.codom == g.codom
    new(f,g)
  end
end
\end{minted}
Given a parallel pair $f,g \maps A \to B$, we must compute the projection map $p \maps B \to C$ onto the object $C$ comprising the coequalizer. The algorithmic idea is to utilize the \textit{union-find} or \textit{disjoint-sets} data structure, which stores an equivalence relation on $B$, and then join the equivalence classes of $f(a)$ and $g(a)$ for all $a \in A$. In the code below, the \jul{IntDisjointSets} data structure stores an equivalence relation on the set of integers $\set{1,\dots,n}$, and the function \jul{union!} efficiently merges two equivalence classes.
\begin{minted}{julia}
function colimit(pair::ParallelPair)
  f, g = pair.f, pair.g
  m, n = f.dom.n, f.codom.n
  sets = IntDisjointSets(n)
  for i in 1:m
    union!(sets, f(i), g(i))
  end
  # Extract the map out of the IntDisjointSets data structure.
  quotient_projection(sets)
end
\end{minted}
For example, the connected components of a graph can be extracted from the coequalizer of the source and target maps $V \to E$. To compute a coequalizer of {\acset} morphisms, one must first have a Julia data structure for {\acset} morphisms. This amounts to a named tuple of morphisms in $\finSetC$, one for each object in the schema. Next, given a parallel pair of {\acset} morphisms, one applies the above colimit function for each object, and then follows the proof that colimits are computed pointwise to construct the colimit {\acset}. Note that this also requires implementing the universal property of the colimit. The details can be found in the Catlab source code.

\subsection{Benchmarks}

In the \cref{subsection:code-generation}, we claimed that the use of generated functions made {\acset}s competitive with hand-written data structures. In this section we provide some evidence for that claim by benchmarking against
\texttt{LightGraphs.jl}, a state-of-the-art Julia package for graphs \cite{bromberger_lightgraphsjl_2017}. LightGraphs boasts performance competitive with graph libraries written in C++ and vastly superior to the popular Python graph library NetworkX \cite{lin_benchmark_2020}.

The results of the benchmarks are shown in \cref{fig:benchmarks}, normalized by the time of the state of the art in Julia (LightGraphs/MetaGraphs). Code to generate these benchmarks is available \href{https://github.com/AlgebraicJulia/Catlab.jl/tree/struct-acsets-benchmarks/benchmark}{on GitHub}. Benchmark results that beat the state of the art are in green, and benchmarks that are worse than 2x the state of the art are in red. In ``Graph'' and ``SymmetricGraph'', we benchmark the operations of checking whether a graph has an edge, iterating through the edges of a graph, iterating through the neighbors of a vertex, and constructing a path graph, for directed and symmetric graphs, respectively. Our data structure for symmetric graphs is intrinsically worse than an ``undirected'' graph data structure, so we cannot expect to be as fast in the symmetric graph benchmarks, unfortunately. In ``GraphConnComponents'' and ``SymmetricGraphConnComponents'', we compute the connected components of path graphs, complete graphs, star graphs, and the Tutte graph. In ``WeightedGraph'' and ``LabeledGraph'', we modify and iterate over the weights and labels of weighted and labeled graphs, respectively. Finally, in ``RandomGraph'', we construct random graphs having various distributions, and in ``Searching'' we traverse the random graphs that we generated previously.

\begin{figure}
\input{figures/benchmark_table.tex}
\caption{Performance Benchmarks for graph algorithms}
\label{fig:benchmarks}
\end{figure}

The benchmarks show that without too much optimization effort, and with no graph-specific code in the {\acset} core, {\acset}s achieve a performance comparable to LightGraphs, often within a factor of two. However, when we compare with \texttt{MetaGraphs.jl}, a commonly used package for attaching data attributes to a LightGraphs graph, we find that {\acset}s provide very large speedups. This is possible because, as we have seen, our implementation of {\acset}s can generate a graph data structure specialized to any particular pattern of vertex and edge attributes. On the other hand, MetaGraphs handles all cases at once by attaching dictionaries to each vertex and  each edge, which results in an inefficient layout of the data in memory.

We conclude from these benchmarks that it is reasonable to use {\acset}s for performance-sensitive tasks in scientific computing. Although further performance gains are surely possible and may be the subject of future work, the system is already usable for most in-memory workloads. Indeed, that {\acset}s achieve fairly good performance in situations where there are viable alternatives (e.g., LightGraphs or DataFrames), as well as in many situations where there are no reasonable alternatives at all, gives us high hopes for the future of this categorical approach to data structures.

\section{Summary and Outlook}

In this paper, we have laid the theoretical and computational groundwork for the use of {\acset}s as a practical data structure for technical computing. The advantages of this approach are threefold. {\Acset}s provide a unifying abstraction for many existing data structures, including graphs and data frames; {\acset}s enable the rapid development of new data structures for relational data; and the category theory underlying {\acset}s is well understood, enabling the implementation of powerful, general operations such as limits, colimits, and functorial data migration.

Many directions for future work remain. One of the most obvious is replacing the category $\setC$ with other categories admitting useful computational representations. For instance, we could replace $\setC$ with $\cat{Par}$, the category of sets and partial maps, which would enable graphs with ``dangling edges'' (edges which may have undefined sources or targets) and, more generally, databases with ``nulls'' (undefined values). The data model based on $\setC$ does not allow nulls, although as a practical matter our implementation allows data in acsets to be temporarily undefined. Other interesting categories besides $\setC$ and $\cat{Par}$ include $\cat{Rel}$, the category of sets and relations; $\cat{Vect}_{\real}$, the category of real vector spaces and linear maps; $\cat{LinRel}_{\real}$, the category of real vector spaces and linear relations; and $\cat{Markov}$, the category of measurable spaces and Markov kernels. For all these large categories, we would work in the skeletization of a finitary subcategory when implementing a data structure, using, for example, the category $\cat{Mat}_{\real}$ of real matrices instead of $\cat{Vect}_{\real}$.

Generalizing in a different direction, we could consider schemas with finite products and product-preserving functors from the schema to $\setC$. The corresponding data structures would generalize from vectors to matrices and higher-order tensors, e.g., the data $F(c_{1} \otimes c_{2} \xrightarrow{f} d)$ would be a $F(c_{1}) \by F(c_{2})$ matrix with values in $F(d)$. This approach would be useful for modeling multivariate observations where both the number of samples and the dimension of each sample are not known in advance, and might be compared with the Python package \texttt{xarray} \cite{hoyer_xarray_2017}. Generalizing further, we could consider schemas with arbitrary finite limits and limit-preserving functors out of them. For example, using pullbacks, one can define a schema whose instances are categories.

All these variants could in principle be implemented in Julia using methods similar to those presented in this paper. However, more theoretical work would be needed to understand the implications for the many categorical constructions supported by {\acset}s.

Another direction for future work would be to improve the symbolic reasoning capabilities around {\acset} mutation. The relations in an {\acset} schema are known to be respected by category-theoretic constructions like limits and colimits but are not automatically verified in user code. For example, when using symmetric graphs (\cref{example:symmetric_graph}), any limit or colimit of symmetric graphs is guaranteed to yield a valid symmetric graph, but it is the user's responsibility to ensure that any direct mutation of the {\acset} data structure preserves the equations governing the edge involution. At least in special circumstances, it should be possible to check whether a given pattern of acset mutation preserves the invariants implied by the relations. This could be done statically, i.e., during the code generation phase, to avoid a runtime penalty. Verification would likely be impractical for arbitrary mutations, but could be carried out for certain basic mutations. Then, provided that all higher-level code used only these basic mutations, the invariants would be guaranteed to hold. For instance, the simple function that adds an edge to a symmetric graph along with its reversal could be checked for correctness. Provided that higher-level code added edges using only this function, the validity of any symmetric graph thus constructed would be guaranteed.

A third direction for future work is creating and manipulating {\acset} schemas using high-level operations. For instance, given a schema $S$, one could construct another schema $S^{\to}$, the ``arrow schema'' on $S$, whose instances are two instances of $S$ together with an {\acset} morphism between them. In principle, it is also possible to take limits and colimits of schemas themselves. Complex data structures could then be constructed compositionally from simpler ones in a principled manner.

To conclude, we might compare the usage of {\acset}s with that of another, highly popular family of parameterized data types: algebraic data types, formalized mathematically as initial algebras of polynomial functors \cite{hamana_gadts_2011}. The {\acset} data model is ubiquitous in its guise as relational databases but rarely implemented as in-memory data structures. On the other hand, many programming languages, especially functional ones, support some form of algebraic data types, yet algebraic data types are quite uncommon in database systems. Given the clear benefits of each paradigm, it is unclear to us why this separation should hold. Our implementation of {\acset}s in Julia, although efficient and usable, would enjoy improved ergonomics if it were a built-in programming language feature; likewise, algebraic data types would be a powerful feature in databases. The two families of data structures are useful in different circumstances: algebraic data types excel at representing syntax trees and other recursive data, while {\acset}s effectively represent graph-like structures and tabular data. Looking forward, we expect to see more implementations of {\acset}s in both programming languages and libraries. We hope that {\acset}s will eventually be recognized as having the same fundamental status as algebraic data types, and will join data frames and graphs as a standard abstraction for data science and scientific computing.

\paragraph{Acknowledgments} The authors thank David Spivak for suggestions regarding the proof of \cref{thm:main_theorem} and Andrew Baas, Kris Brown, Micah Halter, and Sophie Libkind for contributing to the development of AlgebraicJulia and providing useful feedback on the manuscript.

\paragraph{Funding Statement} The authors were supported by DARPA Awards HR00112090067 and HR00111990008, ARO Award W911NF-20-1-0292, and AFOSR Awards FA9550-17-1-0058 and FA9550-20-1-0348. The views and conclusions contained in this document are those of the authors and should not be interpreted as representing the official policies, either expressed or implied, of the Army Research Office or the U.S. Government. The U.S. Government is authorized to reproduce and distribute reprints for Government purposes notwithstanding any copyright notation herein.

\printbibliography[heading=bibintoc]

\end{document}

%% file: figures/port-graphs.tex
\begin{subfigure}[b]{0.5\textwidth}
  \centering
  \includegraphics[width=0.8\textwidth]{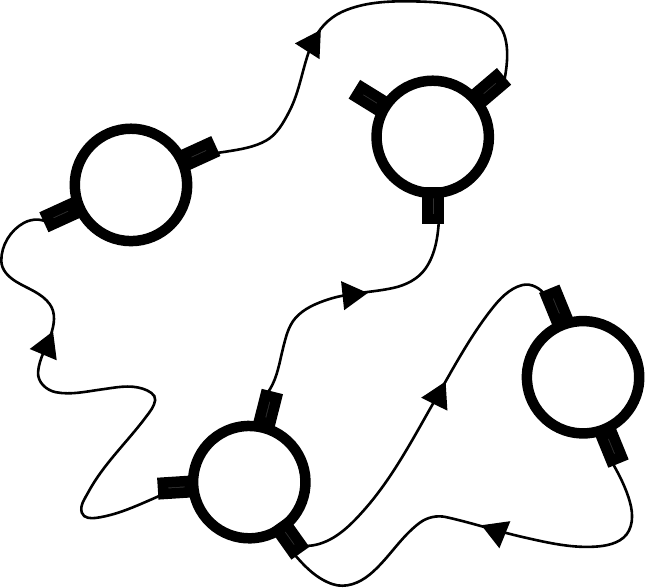}
  \[\begin{tikzcd}
      {\mathbf{Wire}} & {\mathbf{Port}} & {\mathbf{Box}} \\
      {\phantom{mathtt{Type}}}
      \arrow["{\mathrm{box}}"', from=1-2, to=1-3]
      \arrow["{\mathrm{tgt}}"', curve={height=6pt}, from=1-1, to=1-2]
      \arrow["{\mathrm{src}}", curve={height=-6pt}, from=1-1, to=1-2]
    \end{tikzcd}\]
  \caption{Circular port graph}
  \label{subfig:port_graph}
\end{subfigure}%
\begin{subfigure}[b]{0.5\textwidth}
  \centering
  \includegraphics[width=0.8\textwidth]{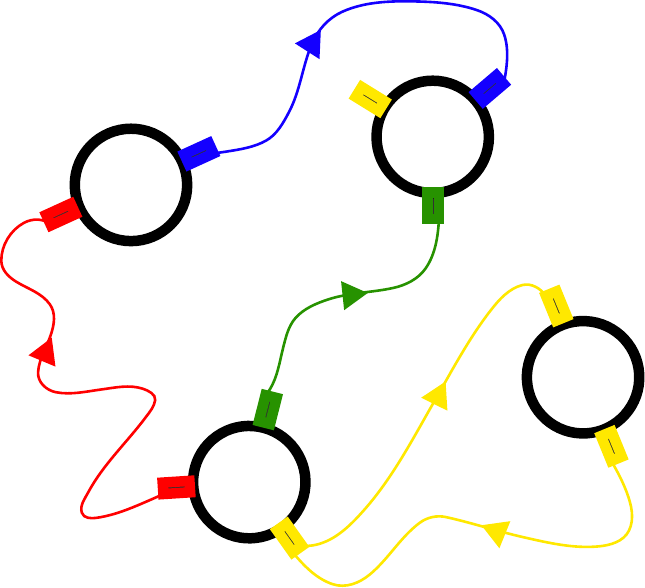}
  \[\begin{tikzcd}
      {\mathbf{Wire}} & {\mathbf{Port}} & {\mathbf{Box}} \\
      {\mathtt{Type}}
      \arrow["{\mathrm{box}}"', from=1-2, to=1-3]
      \arrow["{\mathrm{ptype}}", curve={height=-6pt}, from=1-2, to=2-1]
      \arrow["{\mathrm{wtype}}"', from=1-1, to=2-1]
      \arrow["{\mathrm{tgt}}"', curve={height=6pt}, from=1-1, to=1-2]
      \arrow["{\mathrm{src}}", curve={height=-6pt}, from=1-1, to=1-2]
    \end{tikzcd}\]
  \caption{Typed circular port graph}
  \label{subfig:typed_port_graph}
\end{subfigure}
\begin{subfigure}[b]{0.5\textwidth}
  \centering
  \includegraphics[width=0.8\textwidth]{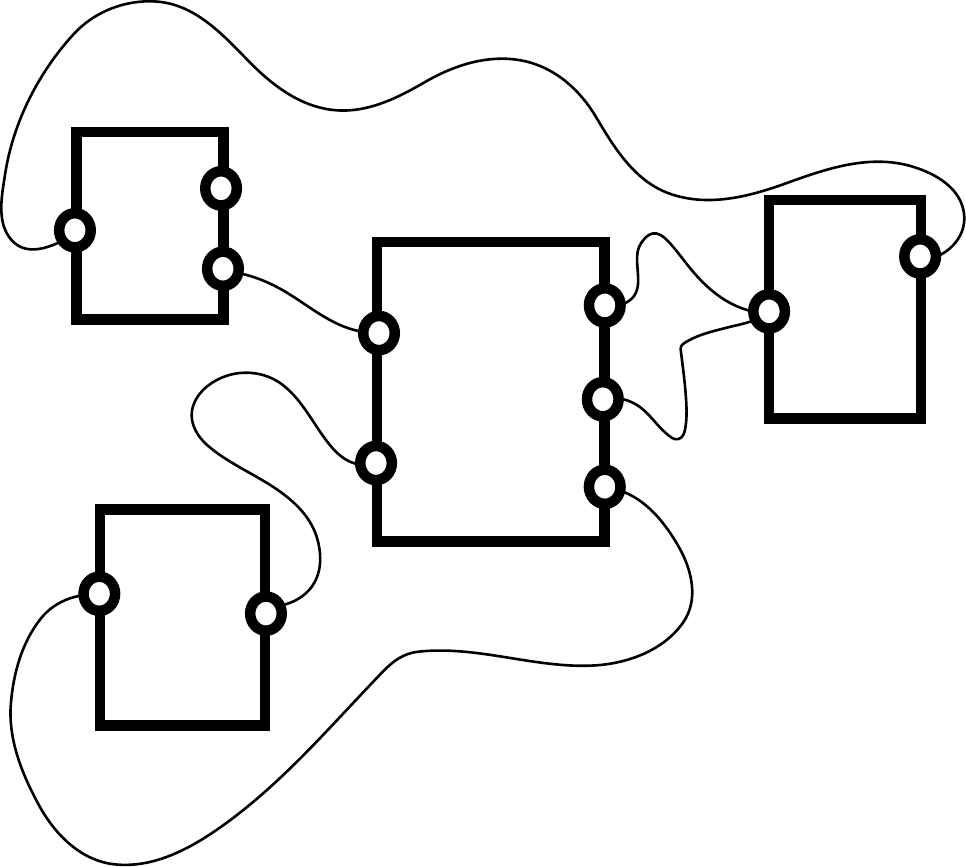}
  \[\begin{tikzcd}
      & {\mathbf{InPort}} \\
      {\mathbf{Wire}} && {\mathbf{Box}} \\
      & {\mathbf{OutPort}}
      \arrow["{\mathrm{src}}"', from=2-1, to=3-2]
      \arrow["{\mathrm{tgt}}", from=2-1, to=1-2]
      \arrow["{\mathrm{ibox}}", from=1-2, to=2-3]
      \arrow["{\mathrm{obox}}"', from=3-2, to=2-3]
    \end{tikzcd}\]
  \caption{Directed port graph}
  \label{subfig:directed_port_graph}
\end{subfigure}%
\begin{subfigure}[b]{0.5\textwidth}
  \centering
  \includegraphics[width=0.8\textwidth]{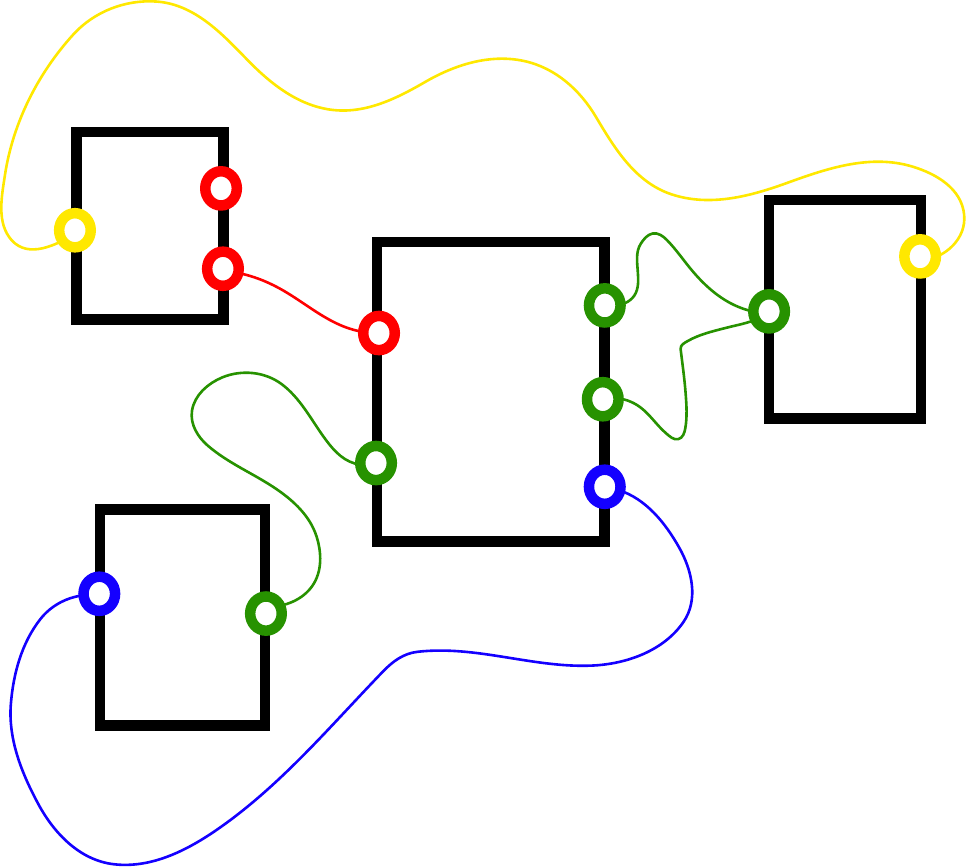}
  \[\begin{tikzcd}
      & {\mathbf{InPort}} \\
      {\mathbf{Wire}} & {\mathtt{Type}} & {\mathbf{Box}} \\
      & {\mathbf{OutPort}}
      \arrow["{\mathrm{src}}"', from=2-1, to=3-2]
      \arrow["{\mathrm{tgt}}", from=2-1, to=1-2]
      \arrow["{\mathrm{ibox}}", from=1-2, to=2-3]
      \arrow["{\mathrm{obox}}"', from=3-2, to=2-3]
      \arrow["{\mathrm{iptype}}"{description}, from=1-2, to=2-2]
      \arrow["{\mathrm{wtype}}", from=2-1, to=2-2]
      \arrow["{\mathrm{optype}}"{description}, from=3-2, to=2-2]
    \end{tikzcd}\]
  \caption{Typed directed port graph}
  \label{subfig:typed_directed_port_graph}
\end{subfigure}

%% file: figures/petri-nets.tex
\begin{subfigure}{0.5\textwidth}
  \centering
  \begin{tikzpicture}[scale=0.7, every place/.style={draw=blue,fill=blue!20,thick,minimum size=9mm}, every transition/.style={draw=red,fill=red!20,thick,minimum size=9mm}]
    \node [place] (R) at (-2,0) {};
    \node [place] (W) at (2,0) {};

    \node [transition] (B) at (-4,0) {}
    edge [pre] (R)
    edge [post,bend right] (R)
    edge [post,bend left] (R);

    \node [transition] (P) at (0,0) {}
    edge [pre] (R)
    edge [pre] (W)
    edge [post,bend left] (W)
    edge [post,bend right] (W);

    \node [transition] (D) at (4,0) {}
    edge [pre] (W);
  \end{tikzpicture}
  \[\begin{tikzcd}[column sep=small]
      & {\mathbf{Input}} \\
      {\phantom{\mathbf{Tok}}} & {\mathbf{Species}} & {\mathbf{Transition}} \\
      {\phantom{\mathtt{Type}}} & {\mathbf{Output}} & {\phantom{\mathtt{Rate}}}
      \arrow[from=3-2, to=2-3]
      \arrow[from=1-2, to=2-3]
      \arrow[from=1-2, to=2-2]
      \arrow[from=3-2, to=2-2]
    \end{tikzcd}\]
  \caption{Petri net}
  \label{fig:petri_net_vanilla}
\end{subfigure}%
\begin{subfigure}{0.5\textwidth}
  \centering
  \begin{tikzpicture}[scale=0.7, every place/.style={draw=blue,fill=blue!20,thick,minimum size=9mm}, every transition/.style={draw=red,fill=red!20,thick,minimum size=9mm}]
    \node [place,tokens=3] (R) at (-2,0) {};
    \node [place,tokens=2] (W) at (2,0) {};

    \node [transition] (B) at (-4,0) {}
    edge [pre] (R)
    edge [post,bend right] (R)
    edge [post,bend left] (R);

    \node [transition] (P) at (0,0) {}
    edge [pre] (R)
    edge [pre] (W)
    edge [post,bend left] (W)
    edge [post,bend right] (W);

    \node [transition] (D) at (4,0) {}
    edge [pre] (W);
  \end{tikzpicture}
  \[\begin{tikzcd}[column sep=small]
      & {\mathbf{Input}} \\
      {\mathbf{Tok}} & {\mathbf{Species}} & {\mathbf{Transition}} \\
      {\phantom{\mathtt{Type}}} & {\mathbf{Output}} & {\phantom{\mathtt{Rate}}}
      \arrow[from=3-2, to=2-3]
      \arrow[from=1-2, to=2-3]
      \arrow[from=1-2, to=2-2]
      \arrow[from=3-2, to=2-2]
      \arrow[from=2-1, to=2-2]
    \end{tikzcd}\]
  \caption{Petri net with tokens}
  \label{fig:petri_net_tokens}
\end{subfigure}
\begin{subfigure}{0.5\textwidth}
  \centering
  \vspace{1em}
  \begin{tikzpicture}[scale=0.7, every place/.style={draw=blue,fill=blue!20,thick,minimum size=9mm}, every transition/.style={draw=red,fill=red!20,thick,minimum size=9mm}]
    \node [place,colored tokens={red,blue,yellow}] (R) at (-2,0) {};
    \node [place,colored tokens={red,red}] (W) at (2,0) {};

    \node [transition] (B) at (-4,0) {}
    edge [pre] (R)
    edge [post,bend right] (R)
    edge [post,bend left] (R);

    \node [transition] (P) at (0,0) {}
    edge [pre] (R)
    edge [pre] (W)
    edge [post,bend left] (W)
    edge [post,bend right] (W);

    \node [transition] (D) at (4,0) {}
    edge [pre] (W);
  \end{tikzpicture}
  \[\begin{tikzcd}[column sep=small]
      & {\mathbf{Input}} \\
      {\mathbf{Tok}} & {\mathbf{Species}} & {\mathbf{Transition}} \\
      {\mathtt{Type}} & {\mathbf{Output}} & {\phantom{\mathtt{Rate}}}
      \arrow[from=3-2, to=2-3]
      \arrow[from=1-2, to=2-3]
      \arrow[from=1-2, to=2-2]
      \arrow[from=3-2, to=2-2]
      \arrow[from=2-1, to=3-1]
      \arrow[from=2-1, to=2-2]
    \end{tikzcd}\]
  \caption{Petri net with typed tokens}
  \label{fig:petri_net_typed_tokens}
\end{subfigure}%
\begin{subfigure}{0.5\textwidth}
  \centering
  \vspace{1em}
  \begin{tikzpicture}[scale=0.7, every place/.style={draw=blue,fill=blue!20,thick,minimum size=9mm}, every transition/.style={draw=red,fill=red!20,thick,minimum size=9mm}]
    \node [place,colored tokens={red,blue,green}] (R) at (-2,0) {};
    \node [place,colored tokens={red,red}] (W) at (2,0) {};

    \node [transition] (B) at (-4,0) {1.0}
    edge [pre] (R)
    edge [post,bend right] (R)
    edge [post,bend left] (R);

    \node [transition] (P) at (0,0) {3.0}
    edge [pre] (R)
    edge [pre] (W)
    edge [post,bend left] (W)
    edge [post,bend right] (W);

    \node [transition] (D) at (4,0) {1.2}
    edge [pre] (W);
  \end{tikzpicture}
  \[\begin{tikzcd}[column sep=small]
      & {\mathbf{Input}} \\
      {\mathbf{Tok}} & {\mathbf{Species}} & {\mathbf{Transition}} \\
      {\mathtt{Type}} & {\mathbf{Output}} & {\mathtt{Rate}}
      \arrow[from=3-2, to=2-3]
      \arrow[from=1-2, to=2-3]
      \arrow[from=1-2, to=2-2]
      \arrow[from=3-2, to=2-2]
      \arrow[from=2-1, to=3-1]
      \arrow[from=2-1, to=2-2]
      \arrow[from=2-3, to=3-3]
    \end{tikzcd}\]
  \caption{Petri net with typed tokens and rates}
  \label{fig:petri_net_typed_tokens_and_rates}
\end{subfigure}

%% file: figures/undirected-wiring-diagrams.tex
\begin{subfigure}[b]{0.5\textwidth}
  \centering
  \includegraphics[width=0.8\textwidth]{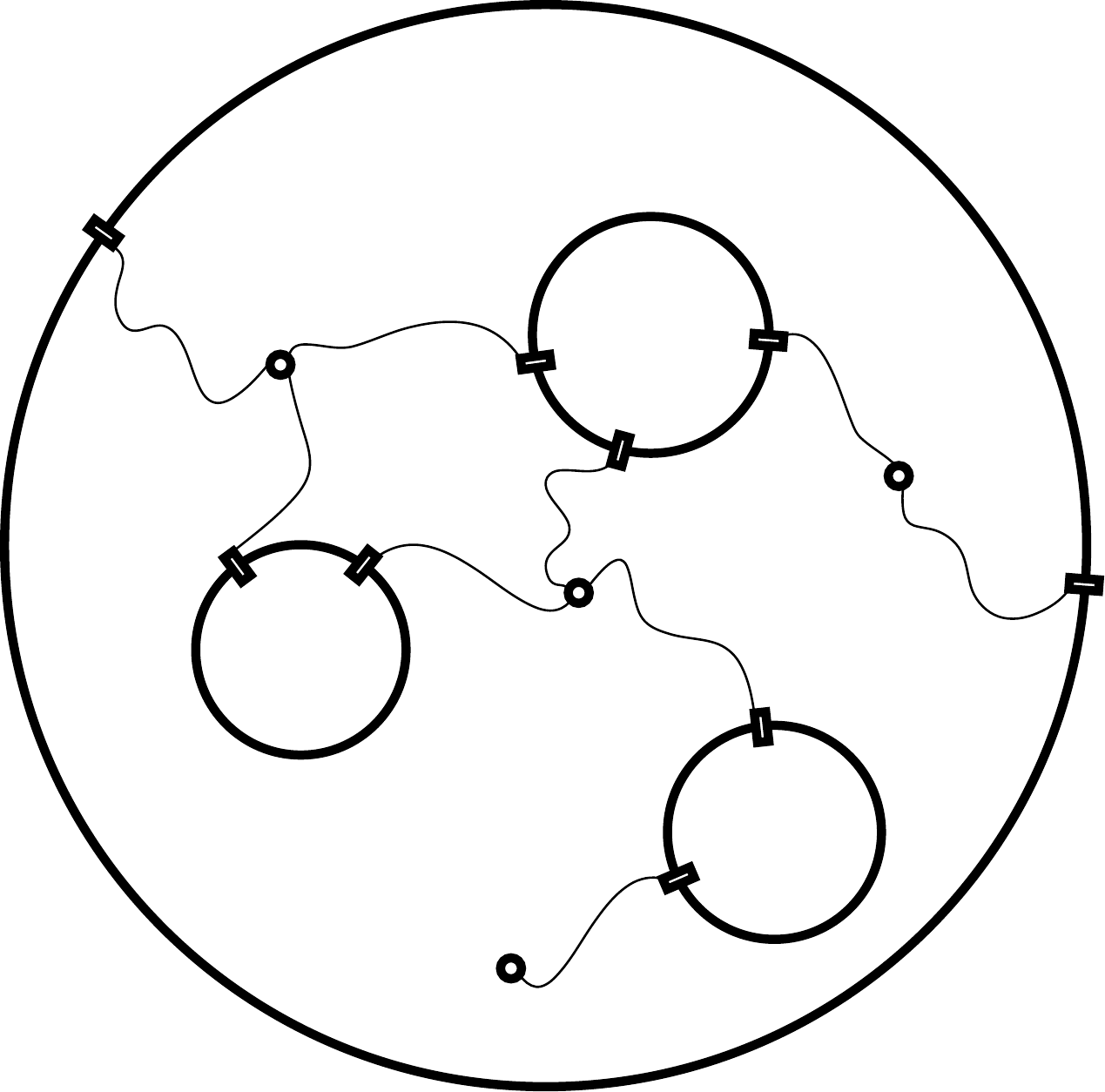}
  \[\begin{tikzcd}
      && {\mathbf{Box}} \\
      {\mathbf{OuterPort}} & {\mathbf{Junction}} & {\mathbf{Port}} \\
      & {\phantom{\mathtt{Type}}}
      \arrow[from=2-3, to=1-3]
      \arrow[from=2-3, to=2-2]
      \arrow[from=2-1, to=2-2]
    \end{tikzcd}\]  \caption{Undirected wiring diagram}
  \label{subfig:undirected_wiring_diagram}
\end{subfigure}%
\begin{subfigure}[b]{0.5\textwidth}
  \centering
  \vspace{2em}
  \includegraphics[width=0.8\textwidth]{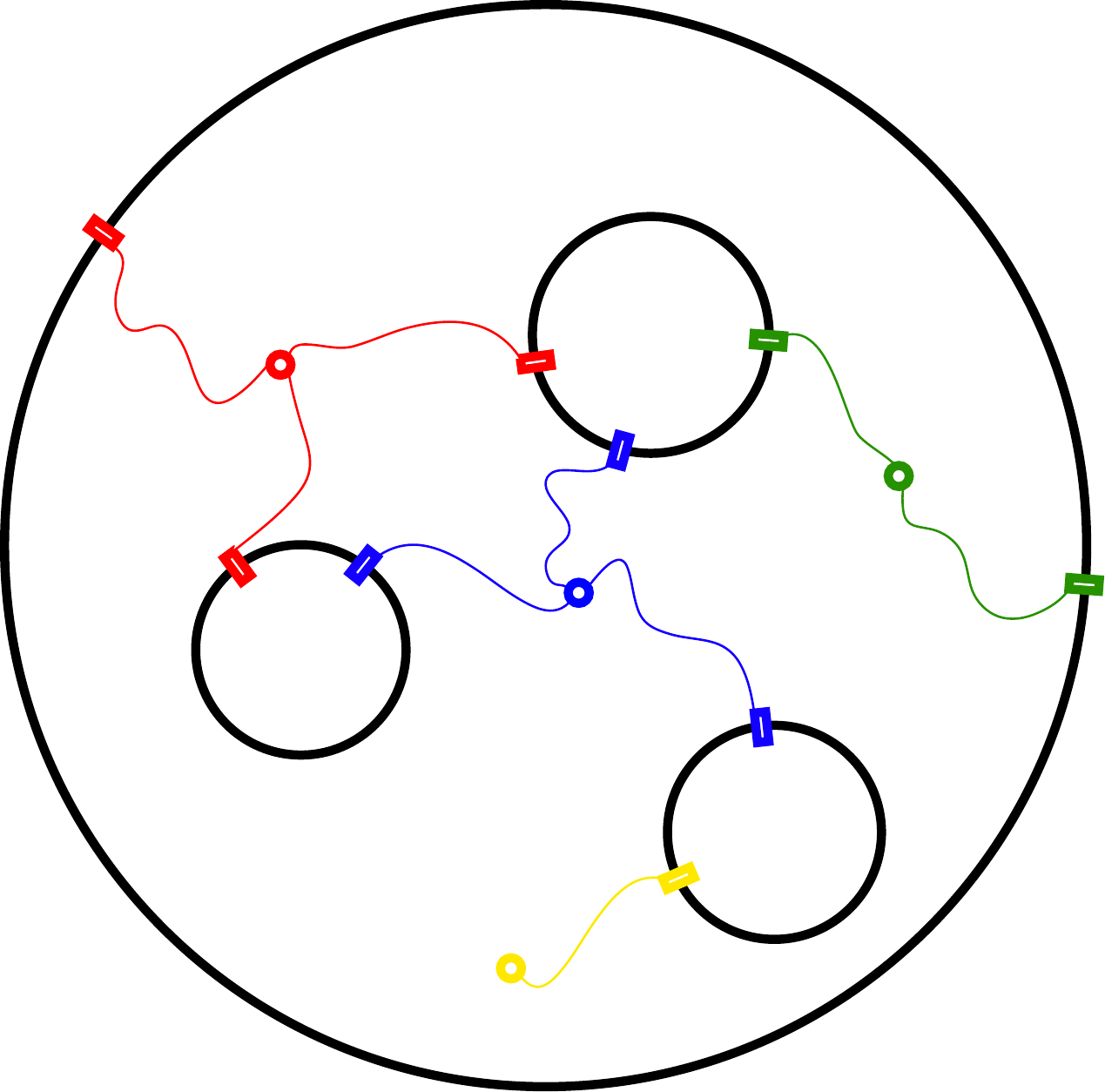}
  \[\begin{tikzcd}
      && {\mathbf{Box}} \\
      {\mathbf{OuterPort}} & {\mathbf{Junction}} & {\mathbf{Port}} \\
      & {\mathtt{Type}}
      \arrow[from=2-3, to=1-3]
      \arrow[from=2-3, to=2-2]
      \arrow[from=2-1, to=2-2]
      \arrow[curve={height=6pt}, from=2-1, to=3-2]
      \arrow[from=2-2, to=3-2]
      \arrow[curve={height=-6pt}, from=2-3, to=3-2]
    \end{tikzcd}\]  \caption{Typed undirected wiring diagram}
  \label{subfig:typed_undirected_wiring_diagram}
\end{subfigure}%

%% file: figures/benchmark_table.tex
\begin{center}
\begin{tabular}{|l|l|r|}
  \hline
  Category & Benchmark & Normalized Time \\ \hline
  Graph & iter-neighbors & 1.37 \\
           & iter-edges & {\color{green} 0.31} \\
           & make-path & {\color{green} 0.31} \\
           & has-edge & {\color{green} 0.54} \\ \hline
  SymmetricGraph & iter-neighbors & {\color{red} 6.19} \\
           & iter-edges & {\color{green} 0.40} \\
           & make-path & {\color{red} 9.53} \\
           & has-edge & {\color{green} 0.78} \\ \hline
  GraphConnComponents & path-graph & {\color{green} 0.43} \\
           & complete100 & {\color{green} 0.04} \\
           & path500 & {\color{green} 0.38} \\
           & star-graph & 1.02 \\ \hline
  SymmetricGraphConnComponents & path-graph-components & 1.11 \\
           & star-graph-components & {\color{green} 0.45} \\
           & complete100 & 1.38 \\
           & path500 & {\color{green} 0.82} \\
           & tutte & 1.21 \\ \hline
  LabeledGraph & indexed-lookup & {\color{green} 0.67} \\
           & make-discrete & {\color{green} 0.72} \\
           & iter-labels & {\color{green} 0.01} \\
           & make-discrete-indexed & {\color{green} 0.45} \\ \hline
  WeightedGraph & sum-weights & {\color{green} 0.001} \\
           & increment-weights & {\color{green} 0.0001} \\ \hline
  RandomGraph & expected\_degree\_graph-10000-10 & 1.55 \\
           & watts\_strogatz-10000-10 & 1.87 \\
           & erdos\_renyi-10000-0.001 & 1.10 \\ \hline
  Searching & dfs\_erdos\_renyi-10000-0.001 & 1.13 \\
           & bfs\_erdos\_renyi-10000-0.001 & 1.25 \\ \hline
\end{tabular}
\end{center}